\documentclass{article}
\usepackage{amsfonts}
\usepackage{amsmath}

\newtheorem{theorem}{Theorem}

\newtheorem{corollary}[theorem]{Corollary}

\newtheorem{definition}[theorem]{Definition}

\newtheorem{lemma}[theorem]{Lemma}

\newtheorem{proposition}[theorem]{Proposition}

\newenvironment{proof}[1][Proof]{\noindent\textbf{#1.} }{\ \rule{0.5em}{0.5em}}
\input{tcilatex}

\begin{document}

\title{Duality between $Y$-convexity and $Y^{\times }$-concavity of linear
operators between Banach lattices}
\author{F. Galaz-Fontes and J.L. Hern\'{a}ndez-Barradas \\
Centro de Investigaci\'{o}n en Matem\'{a}ticas (CIMAT)}
\maketitle

\begin{abstract}
We continue the study of $Y$-convexity, a property which is obtained by
considering a real Banach sequence lattice $Y$ instead of $\ell ^{p}$ for a
linear operator $T:E\rightarrow X,$ where $E$ is a Banach space and $X$ is a
Banach lattice. We introduce some vector sequence spaces in order to
characterize the $Y$-convexity of $T$ by means of the continuity of an
associated operator $\widetilde{T}.$ Analogous results for $Y$-concavity are
also obtained. Finally, the duality between $Y$-convexity and $Y^{\times }$%
-concavity is proven.
\end{abstract}

\section{Introduction}

In \cite{1er art} we present generalizations for $p$-convexity and $q$%
-concavity by considering instead of $\ell ^{p}$ or $\ell ^{q}$ a Banach
sequence lattice $Y$ (see definitions below). In the same paper it is proven
that $Y$\textit{-convexity} and $Y$\textit{-concavity} share many basic
properties with the classical cases. In this work we will establish the
duality relation between $Y$-convexity and $Y$-concavity (see \cite[Ch. 1.d]%
{Linden}). We start by recalling the definitions of $Y$-convexity and $Y$%
-concavity.

\bigskip

Let $\left( \Omega ,\Sigma ,\mu \right) $ be a measure space, where $\Omega $
is a non empty set. Take $L^{0}\left( \mu \right) $ to be the vector space
of (equivalence classes of) measurable functions and denote by $\leq $ the $%
\mu $-almost everywhere order on $L^{0}\left( \mu \right) $. Then a $\mu $%
\textit{-Banach function space} ($\mu $-B.f.s.) $Y$ is an ideal of
measurable functions that is also a Banach\ space with an order-preserving
norm $\left\Vert \cdot \right\Vert _{Y}.$ That is, if $f\in L^{0}\left( \mu
\right) $ and $g\in Y$ are such that $\ \left\vert f\right\vert \leq
\left\vert g\right\vert ,$ then $f\in Y$ and $\left\Vert f\right\Vert
_{Y}\leq \left\Vert g\right\Vert _{Y}$. Notice that a $\mu $-B.f.s. is a%
\textit{\ Banach lattice}. That is, a Riesz space with $\left\vert
f\right\vert =f\vee -f,$ that is also a Banach space with respect to an
order-preserving norm.

\bigskip

Let us now consider the measure space $\left( 
\mathbb{N}
,2^{%
\mathbb{N}
},\mu _{\#}\right) ,$ where $\mu _{\#}$ is the counting measure. If $Y$ is a 
$\mu _{\#}$-B.f.s. and $e_{n}\in Y,n\in 
\mathbb{N}
$, where $e_{n}$ is the $n$th canonical sequence, we say that $Y$ is a 
\textit{Banach sequence lattice (B.s.l}.). For every finite sequence $\left(
t_{1},...,t_{n}\right) \in 
\mathbb{R}
^{n},n\in 
\mathbb{N}
$ we will write $\left( t_{1},...,t_{n}\right) \in Y$ to express the
sequence $\left\{ t_{1},...,t_{n},0,...\right\} \in Y.$ Also, $\left\Vert
\cdot \right\Vert _{%
\mathbb{R}
^{n},Y}$ will indicate the norm of $Y$ restricted to $%
\mathbb{R}
^{n}.$ When $Y=\ell ^{p}$ we will write $\left\Vert \cdot \right\Vert _{%
\mathbb{R}
^{n},p}$ instead of $\left\Vert \cdot \right\Vert _{%
\mathbb{R}
^{n},\ell ^{p}}$. Notice that Orlicz sequence spaces and Lebesgue spaces $%
\ell ^{p},1\leq p\leq \infty $ are B.s.l..

\bigskip

Given $n\in 
\mathbb{N}
,$ let $\mathcal{H}_{n}$ denote the space of all $%
\mathbb{R}
$-valued continuous functions $h$ on $%
\mathbb{R}
^{n}$ which are homogeneous of degree 1, that is, $h(\lambda x)=\lambda h(x)$
for all $x\in 
\mathbb{R}
^{n}$ and $\lambda \geq 0$. Let $X$ be a Banach lattice. Then, vias
Krivine's functional calculus, for every $x=(x_{1},...,x_{n})\in X^{n},$ an
operator $\tau _{n,x}:\mathcal{H}_{n}\rightarrow X$ is defined. So, when
fixing $h\in \mathcal{H}_{n},$ an operator $\widetilde{h}:X^{n}\rightarrow X$
is naturally induced by taking $\widetilde{h}\left( x\right) :=\tau
_{n,x}\left( h\right) .$ When $h\in \mathcal{H}_{n}$ is a norm $\left\Vert
\cdot \right\Vert _{%
\mathbb{R}
^{n},Y},$ we will use the notation $\left\Vert x_{1},...,x_{n}\right\Vert
_{Y}$ instead of $\widetilde{h}\left( x\right) $

\bigskip

Let $Y$ be a B.s.l.. Then, given a Banach lattice $X$ and a Banach space $E$
we will say that a linear operator $T:$ $E\rightarrow X$ is $Y$\textit{%
-convex} if there exists a constant $C>0$ satisfying 
\begin{equation}
\left\Vert \left\Vert \left( Tw_{1},...,Tw_{n}\right) \right\Vert
_{Y}\right\Vert _{X}\leq C\left\Vert \left( \left\Vert w_{1}\right\Vert
_{E},...,\left\Vert w_{n}\right\Vert _{E}\right) \right\Vert _{Y},
\label{20071}
\end{equation}%
for each $n\in 
\mathbb{N}
$ and $w_{1},...,w_{n}\in E.$

\bigskip

\indent The corresponding definition for $Y$-concavity is now clear. A
linear operator $S:X\rightarrow E$ is $Y$\textit{-concave\ }if there exists
a constant $K>0$ satisfying 
\begin{equation}
\left\Vert \left( \left\Vert Sf_{1}\right\Vert _{E},...,\left\Vert
Sf_{n}\right\Vert _{E}\right) \right\Vert _{Y}\leq K\left\Vert \left\Vert
\left( f_{1},...,f_{n}\right) \right\Vert _{Y}\right\Vert _{X},
\label{20072}
\end{equation}%
for each $n\in 
\mathbb{N}
$ and $f_{1},...,f_{n}\in X.$\newline

\bigskip

Observe that when we take $Y=\ell ^{p},1\leq p\leq \infty $ in the above
definitions, conditions $\left( \ref{20071}\right) $ and $\left( \ref{20072}%
\right) $ are reduced to the respective classical $p$-convexity and $p$%
-concavity properties.

\bigskip

This work is divided into 7 sections. After giving an introduction in
Section 1, in Section 2 we present the classical results and definitions
that will be used throughout the paper.

\bigskip

In Section 3 we present some results obtained in \cite{1er art} that will be
needed in the following sections. First we establish that most of the
properties for $\tau _{n,x}\left( \left\Vert \cdot \right\Vert _{%
\mathbb{R}
^{n},p}\right) $ are also valid considering $h=\left\Vert \cdot \right\Vert
_{%
\mathbb{R}
^{n},Y}$. This gives us the function $\left\Vert \cdot \right\Vert
_{X^{n},Y,\tau }:X^{n}\rightarrow 
\mathbb{R}
$ defined by%
\begin{equation*}
\left\Vert \left( x_{1},...,x_{n}\right) \right\Vert _{X^{n},Y,\tau
}:=\left\Vert \tau _{n,x}\left( \left\Vert \cdot \right\Vert _{%
\mathbb{R}
^{n},Y}\right) \right\Vert _{X}.
\end{equation*}%
This function is an order-preserving norm for which $X^{n}$ is a Banach
lattice and that it is equivalent to the natural induced norm $\left\Vert
\cdot \right\Vert _{X^{n},Y}:X^{n}\rightarrow 
\mathbb{R}
$ given by%
\begin{equation*}
\left\Vert \left( x_{1},...,x_{n}\right) \right\Vert _{X^{n},Y}:=\left\Vert
\left( \left\Vert x_{1}\right\Vert _{X},...,\left\Vert x_{n}\right\Vert
_{X}\right) \right\Vert _{Y}.
\end{equation*}%
Then we introduce the concepts of $Y$\textit{-convexity} and $Y$-\textit{%
concavity.}

\bigskip

In Section 4, given a Banach space $E,$ we consider the vector sequence space%
\begin{equation*}
\ell ^{Y}\left( E\right) :=\left\{ \left\{ w_{j}\right\} _{j=1}^{\infty }\in
E^{%
\mathbb{N}
}:\left\{ \left\Vert w_{j}\right\Vert _{E}\right\} _{j=1}^{\infty }\in
Y\right\}
\end{equation*}%
and the function $\left\Vert \cdot \right\Vert _{\ell ^{Y}\left( E\right)
}:\ell ^{Y}\left( E\right) \rightarrow 
\mathbb{R}
$ given by%
\begin{equation*}
\left\Vert w\right\Vert _{\ell ^{Y}\left( E\right) }:=\left\Vert \left\{
\left\Vert w_{j}\right\Vert _{E}\right\} _{j=1}^{\infty }\right\Vert
_{Y},\forall w=\left\{ w_{j}\right\} _{j=1}^{\infty }\in \ell ^{Y}\left(
E\right) .
\end{equation*}%
We prove that $\left\Vert \cdot \right\Vert _{\ell ^{Y}\left( E\right) }$ is
a norm on $\ell ^{Y}\left( E\right) $ and that $\ell ^{Y}\left( E\right) $
is a Banach space. If, in addition, $E$ is a Banach lattice, then $\ell
^{Y}\left( E\right) $ is a Banach lattice.\newline
\newline
\indent Similarly, given a Banach lattice $X,$ we define the vector sequence
lattice \newline
\begin{equation*}
\ell ^{\tau ,Y}\left( X\right) :=\left\{ \left\{ x_{j}\right\}
_{j=1}^{\infty }\in X^{%
\mathbb{N}
}:\sup_{n\in 
\mathbb{N}
}\left\{ \left\Vert \left( x_{1},...,x_{n}\right) \right\Vert _{X^{n},Y,\tau
}\right\} <\infty \right\}
\end{equation*}%
and the function $\left\Vert \cdot \right\Vert _{\ell ^{\tau ,Y}\left(
X\right) }:\ell ^{\tau ,Y}\left( X\right) \rightarrow 
\mathbb{R}
$ given by%
\begin{equation*}
\left\Vert \left\{ x_{j}\right\} _{j=1}^{\infty }\right\Vert _{\ell ^{\tau
,Y}\left( X\right) }:=\sup_{n\in 
\mathbb{N}
}\left\{ \left\Vert \left( x_{1},...,x_{n}\right) \right\Vert _{X^{n},Y,\tau
}\right\} .
\end{equation*}%
We prove that $\left\Vert \cdot \right\Vert _{\ell ^{\tau ,Y}\left( X\right)
}$ is a norm on $\ell ^{\tau ,Y}\left( X\right) $ and that $\ell ^{\tau
,Y}\left( X\right) $ is Banach lattices. \newline
\newline
\indent Let us denote by $\ell _{00}^{Y}\left( E\right) $ the linear
subspace of eventually zero sequences in $\ell ^{Y}\left( E\right) $ and
write $\ell _{0}^{Y}\left( E\right) $ to denote the closure of $\ell
_{00}^{Y}\left( E\right) $ in $\ell ^{Y}\left( E\right) .$ Analogously we
define $\ell _{00}^{\tau ,Y}\left( X\right) $ and $\ell _{0}^{\tau ,Y}\left(
X\right) .$ Then, given a linear operator $T:E\rightarrow X,$ we define the
associated operator $\widetilde{T}:$ $\ell _{00}^{Y}\left( E\right)
\rightarrow \ell _{00}^{\tau ,Y}\left( X\right) ,$ given by $\widetilde{T}%
\left( \left\{ w_{n}\right\} _{n=1}^{\infty }\right) :=\left\{
Tw_{n}\right\} _{n=1}^{\infty }$ and prove that $T$ is $Y$-convex if and
only if $\widetilde{T}$ is continuous.

$\bigskip $

In section 5 we recall \textit{H\"{o}lder's inequality} and \textit{%
associated K\"{o}the space }for a saturated $\mu $-Banach function space $V.$
Then we analyze the case when $V$ is a Banach sequence lattice $Y\ $and
prove that the dual space of $\ell _{0}^{\tau ,Y}\left( X\right) $ is also a
vector sequence space satisfying%
\begin{equation*}
\ell _{0}^{\tau ,Y}\left( X\right) ^{\ast }=\ell ^{\tau ,Y^{\times }}\left(
X^{\ast }\right) .
\end{equation*}

\bigskip

In section 6 we proceed as in Section 5, with the objective to find a
representation for the dual space $\ell _{0}^{\tau ,Y}\left( X\right) ^{\ast
}.$ For this we recall the concept of abstract $L^{1}$-space and some of its
properties. In Lemma \ref{lemaqcv} we generalize an inequality given by
Lindenstrauss and Tzafriri in \cite[Prop. 1.d.2.]{Linden}. Then we prove
that the dual spaces of $\ell _{0}^{Y}\left( E\right) $ is also a vector
sequence space and that%
\begin{equation*}
\ell _{0}^{Y}\left( E\right) ^{\ast }=\ell ^{Y^{\times }}\left( E^{\ast
}\right) ,
\end{equation*}%
\indent Finally in Section 7 we use these results to establish that $Y$%
-convexity and $Y^{\times }$-concavity are dual properties, that is, given a
Banach space $E$ and a Banach lattice $X,$\smallskip \newline
\indent$i)$ An operator $T\in \mathcal{L}\left( E,X\right) $ is $Y$-convex
if and only if the transpose operator $T^{\ast }:X^{\ast }\rightarrow
E^{\ast }$ is $Y^{\times }$-concave$.$\smallskip \newline
\indent$ii)$ An operator $S\in \mathcal{L}\left( X,E\right) \ $is $Y$%
-concave if and only if the transpose operator $S^{\ast }:E^{\ast
}\rightarrow X^{\ast }$ is $Y^{\times }$-convex.\smallskip \newline

\section{Preliminary results and Krivine's funcional calculus}

Throughout this work we will consider only real vector spaces. We will
indicate the \textit{dual pairing} on a normed space $V$ by $\left\langle
\cdot ,\cdot \right\rangle _{V\times V^{\ast }}:V\times V^{\ast }\rightarrow 
\mathbb{R}
.$ Also we will denote by $\mathcal{L}\left( V,W\right) $ the space of
continuous linear operators between two normed spaces $V,W.$ Some of the
definitions and results presented in this section can be consulted in \cite[%
Ch. 1]{Aliprantis} and \cite[Ch. 1 and 2]{Zaanen1}.

\bigskip

Given a Riesz space $X,$ we will denote the\textit{\ positive cone of }$X$
by $X^{+}$ i.e., 
\begin{equation*}
X^{+}:=\left\{ x\in X:x\geq 0\right\} .
\end{equation*}%
\indent Recall a linear operator between Riesz spaces $P:X\rightarrow Z$ is 
\textit{positive} if $P\left( X^{+}\right) \subset Z^{+}$ (see \cite[Ch. 1]%
{Aliprantis})$.$ Also, a linear operator $T:E\rightarrow F$ is \textit{%
regular} if it can be written as a difference of two positive operators. The
space of all regular operators from $X$ into $Z$ is denoted by $\mathcal{L}%
_{r}\left( X,Z\right) .$

\bigskip

Observe that a positive operator $P:X\rightarrow Z$ satisfies%
\begin{equation}
\left\vert Px\right\vert \leq P\left\vert x\right\vert ,\forall x\in X
\label{positive}
\end{equation}%
and, if in addition $X,Z$ are Banach lattices, then $P:X\rightarrow Z$ is
continuous.

\bigskip

Given normed spaces $V,W$ and $T\in \mathcal{L}\left( V,W\right) ,$ we will
indicate the \textit{transpose operator} of $T$ by $T^{\ast }:W^{\ast
}\rightarrow V^{\ast },$ determined by%
\begin{equation*}
\left\langle Tx,\varphi \right\rangle _{W\times W^{\ast }}=\left\langle
x,T^{\ast }\varphi \right\rangle _{V\times V^{\ast }},\forall x\in V,\forall
\varphi \in W^{\ast }.
\end{equation*}%
\indent Recall that $\left\Vert T\right\Vert =\left\Vert T^{\ast
}\right\Vert .$

\bigskip

The following well-known property of $\mu $-B.f.s. will be useful for our
work \cite[Prop. 2.2]{Okada}.

\begin{theorem}
Let $Y$ be a $\mu $-B.f.s., $\left\{ f_{n}\right\} _{n=1}^{\infty }\subset Y$
and $f\in Y.$ If $f_{n}\overset{Y}{\rightarrow }f$, then there exists a
subsequence $\left\{ f_{n_{k}}\right\} _{k=1}^{\infty }$ such that $%
f_{n_{k}}\rightarrow f$ $\mu $-a.e..\label{convmedida}
\end{theorem}

\begin{corollary}
Let $Y\ $be a $\mu $-B.f.s., $f\in L^{0}(\mu )$ and $\left\{ f_{n}\right\}
_{n=1}^{\infty }\subset Y$ a convergent sequence in $Y.$ If $%
f_{n}\rightarrow f$ $\mu $-a.e., then $f\in Y$ and $f_{n}\overset{Y}{%
\rightarrow }f$.\label{coromuctp}
\end{corollary}

\bigskip

Along this work we will use several well-known definitions and properties
about the topological dual, the completion and quotient space of a normed
lattice. Some of these properties are presented below and can be consulted
in \cite[Chapter 1]{Aliprantis} and \cite[Chapter 1]{Linden}.

\bigskip

Now, let $X$ be a Banach lattice. It is well-known that the dual space $%
X^{\ast }$ is a Banach lattice with the order given by 
\begin{equation}
\varphi _{1}\leq \varphi _{2}\text{ if }\varphi _{1}\left( x\right) \leq
\varphi _{2}\left( x\right) ,\forall x\in X^{+}  \label{343}
\end{equation}%
\indent We also have that%
\begin{equation}
\left\vert \varphi \left( x\right) \right\vert \leq \left\vert \varphi
\right\vert \left( \left\vert x\right\vert \right) ,\forall \varphi \in
X^{\ast },x\in X  \label{jijiji}
\end{equation}%
and,\ given $\varphi _{1}...,\varphi _{n}\in X^{\ast }$ the following holds: 
\begin{equation}
\left( \bigvee\limits_{j=1}^{n}\varphi _{j}\right) \left( x\right) =\sup
\left\{ \sum_{j=1}^{n}\varphi _{j}\left( x_{j}\right)
:\sum_{j=1}^{n}x_{j}=x\right\} ,x\in X^{+},0\leq x_{j},j=1,...,n.
\label{defi1802}
\end{equation}

\bigskip

On the other hand, let $E$ be a normed lattice and denote by $\widetilde{E}$
its completion. Then $\widetilde{E}$ is a Banach lattice with the order
given by%
\begin{equation}
u\leq v\text{ if }v-u\in \widetilde{E}^{+},\forall u,v\in \widetilde{E},
\label{defr}
\end{equation}%
where%
\begin{equation}
\widetilde{E}^{+}:=\left\{ w\in \widetilde{E}:\exists \left\{ w_{n}\right\}
_{n=1}^{\infty }\subset E^{+}\text{ s.t. }\left\Vert w-w_{n}\right\Vert
_{E}\rightarrow 0\right\} .  \label{2501}
\end{equation}

\begin{lemma}
Let $X$ be a Riesz space and $Z\subset X$ an ideal. Then the quotient space $%
X/Z$ is a Riesz space with the order given by 
\begin{equation}
\lbrack x]\leq \lbrack y]\text{ if there exists }x_{1}\in \lbrack x]\text{
and }y_{1}\in \lbrack y]\text{ s.t. }x_{1}\leq y_{1}.  \label{344}
\end{equation}
\end{lemma}

A particular case of the above lemma which will be usefull later is when $X$
is a Banach lattice and$\ \varphi \in \left( X^{\ast }\right) ^{+}$. Then%
\begin{equation}
N:=\left\{ x\in X:\varphi \left( \left\vert x\right\vert \right) =0\right\}
\label{qw2}
\end{equation}%
is an ideal and consequently $X/N$ is a Riesz space. Furthermore the
function $\left\Vert \cdot \right\Vert _{\varphi }:X/N\rightarrow 
\mathbb{R}
$ defined by%
\begin{equation}
\left\Vert \lbrack x]\right\Vert _{\varphi }:=\varphi \left( \left\vert
x\right\vert \right) ,\forall x\in X  \label{qw3}
\end{equation}%
is an order-preserving norm on $X/N.$

\QTP{Body Math}
$\bigskip $

\bigskip We now continue describing Krivine's functional calculus (\cite[Ch.
16]{Diestel}, \cite[Theorem 1.d.1]{Linden}, \cite[Lemma 2.51]{Okada}) and
some of its general results.\bigskip

Given $n\in 
\mathbb{N}
,$ recall $\mathcal{H}_{n}$ be the space of continuous functions $h:%
\mathbb{R}
^{n}\rightarrow 
\mathbb{R}
$ that are homogeneous of degree $1$ i.e., for every $\lambda \geq 0,$ $%
h\left( \lambda t_{1},...,\lambda t_{n}\right) =\lambda h\left(
t_{1},...,t_{n}\right) .$ Observe that every norm on $%
\mathbb{R}
^{n}$ is in $\mathcal{H}_{n},$ in particular each $p$-norm, $\left\Vert
\cdot \right\Vert _{%
\mathbb{R}
^{n},p},1\leq p\leq \infty .$

\bigskip

Let us denote by $S^{n-1}\subset 
\mathbb{R}
^{n}$ the unit sphere with respect to the supremum norm $\left\Vert \cdot
\right\Vert _{%
\mathbb{R}
^{n},\infty },$ given by $\left\Vert \left( a_{1},...,a_{n}\right)
\right\Vert _{%
\mathbb{R}
^{n},\infty }:=\max_{1\leq j\leq n}\left\{ \left\vert a_{j}\right\vert
\right\} ,$ $\left( a_{1},...,a_{n}\right) \in 
\mathbb{R}
^{n}.$ Then we define the norm $\left\Vert \cdot \right\Vert _{\mathcal{H}%
_{n}}$ on $\mathcal{H}_{n}$ by $\left\Vert h\right\Vert _{\mathcal{H}%
_{n}}:=\sup \left\{ \left\vert h\left( t\right) \right\vert :t\in
S^{n-1}\right\} .$ Thus, with the pointwise order, we have that $\mathcal{H}%
_{n}$ is a Banach lattice.

\bigskip

For each $n\in 
\mathbb{N}
$ and $1\leq j\leq n$ we define the $j$th projection $\pi _{n,j}:%
\mathbb{R}
^{n}\rightarrow 
\mathbb{R}
$ by $\pi _{n,j}\left( t_{1},...,t_{n}\right) =t_{j}.$ Clearly $\pi
_{n,j}\in \mathcal{H}_{n},j=1,...,n.$

\begin{theorem}
(Krivine's functional calculus) Let $X$ be a Banach lattice, $n\in 
\mathbb{N}
$ and $x=(x_{1},...,x_{n})\in X^{n}.$ Then, there exists a unique linear map 
$\tau _{n,x}:\label{krivin}$ $\mathcal{H}_{n}\rightarrow X$ such
that\smallskip \newline
\indent$i)$ ${\large \tau }_{n,x}\left( \pi _{n,j}\right) {\large =x}_{j}$
for $1\leq j\leq n.$ \smallskip \newline
\indent$ii)$ ${\Large \tau }_{n,x}$ is order-preserving i.e.,%
\begin{equation*}
\tau _{n,x}\left( h_{1}\vee h_{2}\right) =\tau _{n,x}\left( h_{1}\right)
\vee \tau _{n,x}\left( h_{2}\right) ,\forall h_{1},h_{2}\in \mathcal{H}_{n}.
\end{equation*}
\end{theorem}

\bigskip

Given $h\in $ $\mathcal{H}_{n}$, Krivine's functional calculus determines
the function $\widetilde{h}:X^{n}\rightarrow X$ given by $\widetilde{h}(x):=$
${\large \tau }_{n,x}\left( h\right) $. When there is no risk of confusion
we will write $h$ instead of $\widetilde{h}.$ We will now analyze the
function $h:X^{n}\rightarrow X.$

\begin{definition}
For $n,m\in 
\mathbb{N}
$ we define%
\begin{equation*}
\mathcal{H}_{n}^{m}:=\left\{ G:%
\mathbb{R}
^{n}\rightarrow 
\mathbb{R}
^{m}\text{ s.t. }g_{j}\in \mathcal{H}_{n},\forall \text{ }1\leq j\leq
m\right\} ,
\end{equation*}%
where $g_{j}=\pi _{m,j}\circ G\mathfrak{.}$
\end{definition}

\bigskip

Given $G=\left( g_{1},...,g_{m}\right) \in \mathcal{H}_{n}^{m}$ we have that%
\begin{equation*}
G\left( x\right) :=\left( g_{1}\left( x\right) ,...,g_{m}\left( x\right)
\right) ,\forall x\in X^{n}.
\end{equation*}%
\indent Note that $\mathcal{H}_{n}^{m}$ is a vector space, on which we
define the norm $\left\Vert \cdot \right\Vert _{\mathcal{H}_{n}^{m}}$ by%
\begin{equation*}
\left\Vert G\right\Vert _{\mathcal{H}_{n}^{m}}:=\max_{1\leq j\leq m}\left\{
\left\Vert g_{j}\right\Vert _{\mathcal{H}_{n}}\right\} ,\forall G\in 
\mathcal{H}_{n}^{m}.
\end{equation*}

Let $G\in \mathcal{H}_{n}^{m}.$ Then for each $h\in \mathcal{H}_{m}$ we can
consider the composition $h\circ G:%
\mathbb{R}
^{n}\rightarrow 
\mathbb{R}
$. So we define the operator $J_{G}$ by%
\begin{equation*}
J_{G}\left( h\right) :=h\circ G,\forall h\in \mathcal{H}_{m}.
\end{equation*}%
It follows that 
\begin{equation}
J_{G}\left( h\right) \in \mathcal{H}_{n}.  \label{smt2}
\end{equation}%
Furthermore, given $h_{1},h_{2}\in \mathcal{H}_{m}$ and $\lambda \in 
\mathbb{R}
$ 
\begin{equation*}
J_{G}\left( \lambda h_{1}+h_{2}\right) =\left( \lambda h_{1}+h_{2}\right)
\circ G=\lambda \left( h_{1}\circ g\right) +h_{2}\circ G=\lambda J_{G}\left(
h_{1}\right) +J_{G}\left( h_{2}\right)
\end{equation*}%
and%
\begin{equation*}
J_{G}\left( h_{1}\vee h_{2}\right) =\left( h_{1}\vee h_{2}\right) \circ
G=\left( h_{1}\circ G\right) \vee \left( h_{2}\circ G\right) =J_{G}\left(
h_{1}\right) \vee J_{G}\left( h_{2}\right) .
\end{equation*}%
\indent Hence, $J_{G}:\mathcal{H}_{m}\rightarrow \mathcal{H}_{n}$ is an
order-preserving linear operator.

\bigskip

The next results are proven in \cite{1er art}.

\begin{lemma}
Let $n,m\in 
\mathbb{N}
$ and $G\in \mathcal{H}_{n}^{m}.$ Then, for each $x=\left(
x_{1},...,x_{n}\right) \in X^{n},$\label{lematmm}\newline
\begin{equation}
{\large \tau }_{m,G\left( x\right) }={\large \tau }_{n,x}\circ J_{G}.
\label{3agos}
\end{equation}
\end{lemma}

\begin{lemma}
Let $X$, $W$ be Banach lattices and $T:X\rightarrow W$ an order-preserving
linear operator. Then, for each $n\in 
\mathbb{N}
$ and $x_{1},...,x_{n}\in X,\label{lema1802}$%
\begin{equation}
T\circ \tau _{\left( x_{1,}...,x_{n}\right) }=\tau _{\left(
Tx_{1},...,Tx_{n}\right) }  \label{fr1}
\end{equation}
\end{lemma}

The next proposition is established in \cite[Theorem 1.d.1]{Linden}.

\begin{proposition}
Let $X$ be a Banach lattice, $n\in 
\mathbb{N}
$ and $x=(x_{1},...,x_{n})\in X^{n}.$ Then, for each $h\in \mathcal{H}_{n},%
\label{propkrivi}$%
\begin{equation}
\tau _{n,x}\left( h\right) \leq \left\Vert h\right\Vert _{\mathcal{H}%
_{n}}\max_{1\leq j\leq n}\left\{ \left\vert x_{j}\right\vert \right\} \in X.
\label{nuevaprop}
\end{equation}%
\indent Therefore%
\begin{equation}
\left\Vert \tau _{n,x}\left( h\right) \right\Vert _{X}\leq \left\Vert
h\right\Vert _{\mathcal{H}_{n}}\left\Vert \max_{1\leq j\leq n}\left\{
\left\vert x_{j}\right\vert \right\} \right\Vert _{X}.  \label{NUEVAPROP2}
\end{equation}
\end{proposition}

\section{Krivine's functional calculus for $\left\Vert \cdot \right\Vert _{%
\mathbb{R}
^{n},Y}$}

From now on we fix $Y$ to be a Banach sequence lattice and $X,W$ to be
Banach lattices. Notice that, for each $n\in 
\mathbb{N}
,$ $\left\Vert \cdot \right\Vert _{%
\mathbb{R}
^{n},Y}\in \mathcal{H}_{n}.$ Then, for each $x=\left( x_{1},...,x_{n}\right)
\in X^{n}$ we define $\left\Vert \cdot \right\Vert _{Y}:X^{n}\rightarrow X$
by%
\begin{equation*}
\left\Vert \left( x_{1},...,x_{n}\right) \right\Vert _{Y}:=\tau _{n,x}\left(
\left\Vert \cdot \right\Vert _{%
\mathbb{R}
^{n},Y}\right) .
\end{equation*}%
In this section we analyze some properties of the above operator and use it
to construct an order-preserving norm on $X^{n}$. Proofs for the following
results can be consulted in \cite{1er art}.

\bigskip

\begin{lemma}
Let $X$, $W$ be Banach lattices\smallskip .\label{lemalema}\newline
\indent$i)$ If $n\in 
\mathbb{N}
,$ $\lambda \in 
\mathbb{R}
$ and $x=\left( x_{1},...,x_{n}\right) \in X^{n},$ then%
\begin{equation*}
\left\Vert \left( \lambda x_{1},...,\lambda x_{n}\right) \right\Vert
_{Y}=\left\vert \lambda \right\vert \left\Vert \left( x_{1},...,x_{n}\right)
\right\Vert _{Y}.
\end{equation*}%
\indent$ii)$ If $x=\left( x_{1},...,x_{n}\right) ,y=\left(
y_{1},...,y_{n}\right) \in X^{n}$ and $\left\vert x_{j}\right\vert \leq
\left\vert y_{j}\right\vert ,$ $1\leq j\leq n,$ then\label{propo2305}%
\begin{equation*}
\left\Vert \left( x_{1},...,x_{n}\right) \right\Vert _{Y}\leq \left\Vert
\left( y_{1},...,y_{n}\right) \right\Vert _{Y}\in X
\end{equation*}%
and as a consequence%
\begin{equation*}
\left\Vert \left( x_{1},...,x_{n}\right) \right\Vert _{Y}=\left\Vert \left(
\left\vert x_{1}\right\vert ,...,\left\vert x_{n}\right\vert \right)
\right\Vert _{Y}.
\end{equation*}%
\indent$iii)$ If $x=\left( x_{1},...,x_{n}\right) $ and $y=\left(
y_{1},...,y_{n}\right) \in X^{n}$, then\label{lema23052}%
\begin{equation*}
\left\Vert \left( x_{1}+y_{1},...,x_{n}+y_{n}\right) \right\Vert _{Y}\leq
\left\Vert \left( x_{1},...,x_{n}\right) \right\Vert _{Y}+\left\Vert \left(
y_{1},...,y_{n}\right) \right\Vert _{Y}\in X.
\end{equation*}%
\indent Therefore%
\begin{equation}
\left\Vert \left\Vert \left( x_{1}+y_{1},...,x_{n}+y_{n}\right) \right\Vert
_{Y}\right\Vert _{X}\leq \left\Vert \left\Vert \left( x_{1},...,x_{n}\right)
\right\Vert _{Y}\right\Vert _{X}+\left\Vert \left\Vert \left(
y_{1},...,y_{n}\right) \right\Vert _{Y}\right\Vert _{X}.  \label{835}
\end{equation}%
\indent$iv)$ If $x=\left( x_{1},...,x_{n}\right) \in X^{n},$ then\label%
{lema2.51vi} 
\begin{equation*}
\left\Vert \left( x_{1},...,x_{n}\right) \right\Vert _{Y}\leq \left\Vert
e_{1}+...+e_{n}\right\Vert _{%
\mathbb{R}
^{n},Y}\max_{1\leq j\leq n}\left\{ \left\vert x_{j}\right\vert \right\} \in X
\end{equation*}%
and as a consequence%
\begin{equation*}
\left\Vert \left\Vert \left( x_{1},...,x_{n}\right) \right\Vert
_{Y}\right\Vert _{X}\leq \left\Vert e_{1}+...+e_{n}\right\Vert _{%
\mathbb{R}
^{n},Y}\left\Vert \max_{1\leq j\leq n}\left\{ \left\vert x_{j}\right\vert
\right\} \right\Vert _{X}.
\end{equation*}%
\newline
\indent$v)$ Let $T:X\rightarrow W$ be an order-preserving linear operator.
Then%
\begin{equation*}
T\left( \left\Vert \left( x_{1},...,x_{n}\right) \right\Vert _{Y}\right)
=\left\Vert \left( Tx_{1},...,Tx_{n}\right) \right\Vert _{Y},n\in 
\mathbb{N}
,x_{1},...,x_{n}\in X
\end{equation*}
\end{lemma}

\bigskip

Given a Banach space $E$ and $n\in 
\mathbb{N}
$, we will write $\left\Vert \cdot \right\Vert _{E^{n},Y}$ to denote the
norm on $E^{n}$ induced naturally by $\left\Vert \cdot \right\Vert _{%
\mathbb{R}
^{n},Y}$. That is,\newline
\begin{equation}
\left\Vert w\right\Vert _{E^{n},Y}:=\left\Vert \left( \left\Vert
w_{1}\right\Vert _{E},...,\left\Vert w_{n}\right\Vert _{E}\right)
\right\Vert _{%
\mathbb{R}
^{n},Y},\forall w=\left( w_{1},...,w_{n}\right) \in E^{n}.
\label{norma_X^n_Y}
\end{equation}

Note that the norms $\left\Vert \cdot \right\Vert _{E^{n},1}$ and $%
\left\Vert \cdot \right\Vert _{E^{n},Y}$ are equivalent since they are
defined by the equivalent norms $\left\Vert \cdot \right\Vert _{%
\mathbb{R}
^{n},1}$ and $\left\Vert \cdot \right\Vert _{%
\mathbb{R}
^{n},Y}$ respectively. Then, as $\left( E^{n},\left\Vert \cdot \right\Vert
_{E^{n},1}\right) $ is complete it follows that $\left( E^{n},\left\Vert
\cdot \right\Vert _{E^{n},Y}\right) $ is a Banach space. Now, instead of a
Banach space $E$ consider a Banach lattice $X.$ Then $X^{n}$ is a Banach
lattice with the order by components. Furthermore, for every $x=\left(
x_{1},...,x_{n}\right) $, $z=\left( z_{1},...,z_{n}\right) \in X^{n}$ such
that $x\leq z,$%
\begin{equation*}
\left\Vert x\right\Vert _{X^{n},Y}=\left\Vert \left( \left\Vert
x_{1}\right\Vert _{X},...,\left\Vert x_{n}\right\Vert _{X}\right)
\right\Vert _{Y}\leq \left\Vert \left( \left\Vert z_{1}\right\Vert
_{X},...,\left\Vert z_{n}\right\Vert _{X}\right) \right\Vert _{Y}=\left\Vert
z\right\Vert _{X^{n},Y}.
\end{equation*}%
Hence, in this case $\left\Vert \cdot \right\Vert _{X^{n},Y}$ is an
order-preserving norm and consequently\linebreak\ $\left( X^{n},\left\Vert
\cdot \right\Vert _{X^{n},Y}\right) $ is a Banach lattice.

\bigskip

Next, making use of Lemma \ref{lemalema}, we define another norm on the
lattice $X^{n}$. This will be useful to prove the completeness of the spaces
of $Y$-convex and $Y$-concave operators, which we will consider later. The
relation of this norm\ with $\left\Vert \cdot \right\Vert _{X^{n},Y}$ will
be analyzed in the next section.\newline

\begin{definition}
For each $n\in 
\mathbb{N}
$ we define the function $\left\Vert \cdot \right\Vert _{X^{n},Y,\tau
}:X^{n}\rightarrow 
\mathbb{R}
$ by%
\begin{equation*}
\left\Vert x\right\Vert _{X^{n},Y,\tau }:=\left\Vert \left\Vert \left(
x_{1},...,x_{n}\right) \right\Vert _{Y}\right\Vert _{X},\forall x=\left(
x_{1},...,x_{n}\right) \in X^{n}.
\end{equation*}
\end{definition}

Observe that, by Lemma \ref{lemalema}, $\left\Vert \cdot \right\Vert
_{X^{n},Y,\tau }$ is a lattice norm.

\begin{proposition}
The norms $\left\Vert \cdot \right\Vert _{X^{n},1}$ and $\left\Vert \cdot
\right\Vert _{X^{n},Y,\tau }$ are equivalent and consequently $\left(
X^{n},\left\Vert \cdot \right\Vert _{X^{n},Y,\tau }\right) $ is a Banach
lattice.\label{propoca}
\end{proposition}

\bigskip

The following properties are straightforward implications of Lemma \ref%
{lemalema} and the above proposition.

\bigskip

Let $X$ be a Banach lattice and $\left\{ x_{m}\right\} _{m=1}^{\infty
}\subset X^{n}.$ Then$\smallskip $\newline
\indent$i)$ The sequence $\left\{ x_{m}\right\} _{m=1}^{\infty }$ is a
Cauchy sequence in $\left( X^{n},\left\Vert \cdot \right\Vert
_{X^{n},Y}\right) $ if and only if it is a Cauchy sequence in $X$ in each
one of its components.$\smallskip $\newline
\indent$ii)$ The sequence $\left\{ x_{m}\right\} _{m=1}^{\infty }$ is
convergent in $\left( X^{n},\left\Vert \cdot \right\Vert _{X^{n},Y}\right) $
if and only if it is convergent in $X$ in each one of its components. In
this case for $x_{m}:=\left( x_{m,1},...,x_{m,n}\right) ,m\in 
\mathbb{N}
$ we have that%
\begin{equation*}
\lim_{m\rightarrow \infty }x_{m}=\left( \lim_{m\rightarrow \infty
}x_{m,1},...,\lim_{m\rightarrow \infty }x_{m,n}\right)
\end{equation*}%
and%
\begin{equation*}
\left\Vert \lim_{m\rightarrow \infty }x_{m}\right\Vert _{X^{n},Y}=\left\Vert
\left( \lim_{m\rightarrow \infty }\left\Vert x_{m,1}\right\Vert
_{X},...,\lim_{m\rightarrow \infty }\left\Vert x_{m,n}\right\Vert
_{X}\right) \right\Vert _{Y}.
\end{equation*}%
\indent$iii)$ For each $n\in 
\mathbb{N}
,$%
\begin{equation*}
\left\Vert \left( x_{1},...,x_{n},0\right) \right\Vert
_{X^{n+1},Y}=\left\Vert \left( x_{1},...,x_{n}\right) \right\Vert _{X^{n},Y},
\end{equation*}

\bigskip

\begin{equation*}
\left\Vert \left( x_{1},...,x_{n},0\right) \right\Vert _{X^{n+1},Y,\tau
}=\left\Vert \left( x_{1},...,x_{n}\right) \right\Vert _{X^{n},Y,\tau },
\end{equation*}

\bigskip

\begin{equation*}
\left\Vert \left( x_{1},...,x_{n}\right) \right\Vert _{X^{n},Y}\leq
\left\Vert \left( x_{1},...,x_{n},x_{n+1}\right) \right\Vert _{X^{n+1},Y},
\end{equation*}%
and\newline

\begin{equation*}
\left\Vert \left( x_{1},...,x_{n}\right) \right\Vert _{X^{n},Y,\tau }\leq
\left\Vert \left( x_{1},...,x_{n},x_{n+1}\right) \right\Vert
_{X^{n+1},Y,\tau }.
\end{equation*}

\bigskip

Next we will present a characterization of $\tau _{n,x}\left( \left\Vert
\cdot \right\Vert _{%
\mathbb{R}
^{n},Y}\right) .$ In order to do this, we will identify the dual space $%
\left( 
\mathbb{R}
^{n},\left\Vert \cdot \right\Vert _{Y}\right) ^{\ast }$ with $\left( 
\mathbb{R}
^{n},\left\Vert \cdot \right\Vert _{Y^{\ast }}\right) .$ That is, we will
identify each functional $\varphi \in \left( 
\mathbb{R}
^{n},\left\Vert \cdot \right\Vert _{Y}\right) ^{\ast }$ with the unique
vector $\left( a_{1},...,a_{n}\right) \in 
\mathbb{R}
^{n}$ such that%
\begin{equation*}
\left\langle \left( t_{1},...,t_{n}\right) ,\varphi \right\rangle
=\sum_{j=1}^{n}a_{j}t_{j},\forall \left( t_{1},...,t_{n}\right) \in 
\mathbb{R}
^{n},
\end{equation*}%
and%
\begin{equation*}
\left\Vert \left( a_{1},...,a_{n}\right) \right\Vert _{Y^{\ast }}:=\sup
\left\{ \left\vert \sum_{j=1}^{n}a_{j}t_{j}\right\vert :\left\Vert \left(
t_{1},...,t_{n}\right) \right\Vert _{%
\mathbb{R}
^{n},Y}\leq 1\right\} .
\end{equation*}%
\newline
\indent Thus, for each $t_{1},...,t_{n}\in 
\mathbb{R}
^{n}$ we have that%
\begin{eqnarray}
\left\Vert \left( t_{1},...,t_{n}\right) \right\Vert _{Y} &=&\sup \left\{
\left\vert \left\langle \left( t_{1},...,t_{n}\right) ,\varphi \right\rangle
\right\vert :\left\Vert \varphi \right\Vert _{\left( 
\mathbb{R}
^{n},\left\Vert \cdot \right\Vert _{Y}\right) ^{\ast }}\leq 1\right\}  \notag
\\
&=&\sup \left\{ \left\vert \sum_{j=1}^{n}a_{j}t_{j}\right\vert :\left\Vert
\left( a_{1},...,a_{n}\right) \right\Vert _{%
\mathbb{R}
^{n},Y^{\ast }}\leq 1\right\} .  \label{7nov}
\end{eqnarray}

\begin{lemma}
Let $K$ be a compact space, $\mathcal{C}\subset C\left( K\right) $ a non
empty family and $g\in C\left( K\right) .$ If $g=\sup \mathcal{C}$, then,
there exists a sequence $\left\{ g_{n}\right\} _{n=1}^{\infty }\subset
C\left( K\right) $ such that 
\begin{equation}
\left\Vert f-g_{n}\right\Vert _{C\left( K\right) }\rightarrow 0  \label{jui}
\end{equation}%
and each $g_{n}$ is a maximum of functions on \label{2555}$\mathcal{C}.$
\end{lemma}

\bigskip

The next result extends that of the case $Y=\ell ^{p},1\leq p\leq \infty $ 
\cite[p. 42]{Linden}$.$

\begin{proposition}
For each $n\in 
\mathbb{N}
$ and $x=\left( x_{1},...,x_{n}\right) \in X^{n},\label{lema69}\qquad $%
\begin{equation}
\left\Vert \left( x_{1},...,x_{n}\right) \right\Vert _{Y}=\sup \left\{
\sum_{j=1}^{n}a_{j}x_{j}:\left\Vert \left( a_{1},...,a_{n}\right)
\right\Vert _{Y^{\ast }}\leq 1\right\} .  \label{muero3}
\end{equation}
\end{proposition}

\begin{corollary}
For each $n\in 
\mathbb{N}
$ and $x=\left( x_{1},...,x_{n}\right) \in X^{n},\label{coro1405}$%
\begin{equation}
\left\Vert x\right\Vert _{X^{n},Y,\tau }\leq \sup_{k\in 
\mathbb{N}
}\left\{ \left\Vert
\bigvee\limits_{j=1}^{k}\sum_{i=1}^{n}a_{i,j}x_{i}\right\Vert
_{X}:\left\Vert \left( a_{1,j},...,a_{n,j}\right) \right\Vert _{Y^{\ast
}}\leq 1,1\leq j\leq k\right\} .  \label{wsde}
\end{equation}
\end{corollary}

\bigskip

Next we present our generalization of $p$-convexity and $q$-concavity:

\begin{definition}
A linear operator $T:$ $E\rightarrow X$ is $Y$-convex if there exists a
constant $C>0$ satisfying 
\begin{equation}
\left\Vert \left( Tw_{1},...,Tw_{n}\right) \right\Vert _{X^{n},Y,\tau }\leq
C\left\Vert \left( w_{1},...,w_{n}\right) \right\Vert _{E^{n},Y},
\label{20073}
\end{equation}%
for each $n\in 
\mathbb{N}
$ and $w_{1},...,w_{n}\in E.$ The smallest constant satisfying (\ref{20073})
for all such $n\in N$ and $w_{j}$'s ($j=1,...,n$) is called the $Y$%
-convexity constant of $T$ and is denoted by $\left\Vert T\right\Vert
_{K^{Y}}.$ Also, we will write%
\begin{equation*}
\mathcal{K}^{Y}\left( E,X\right) =\left\{ T:E\rightarrow X\text{ s.t. }T%
\text{ is }Y\text{-convex}\right\} .
\end{equation*}%
\newline
\indent Similarly, a linear operator $S:X\rightarrow E$ is $Y$\textit{%
-concave\ }if there exists a constant $K>0$ satisfying 
\begin{equation}
\left\Vert \left( Sx_{1},...,Sx_{n}\right) \right\Vert _{E^{n},Y}\leq
K\left\Vert \left( x_{1},...,x_{n}\right) \right\Vert _{X^{n},Y,\tau },
\label{20074}
\end{equation}%
for each $n\in 
\mathbb{N}
$ and $x_{1},...,x_{n}\in X.$ The smallest constant satisfying (\ref{20074})
for all such $n\in N$ and $x_{j}$'s ($j=1,...,n$) is called the $Y$%
-concavity constant of $S$ and is denoted by $\left\Vert S\right\Vert
_{K_{Y}}.$ Also, we will write%
\begin{equation*}
\mathcal{K}_{Y}\left( X,E\right) =\left\{ S:X\rightarrow E\text{ s.t. }S%
\text{ is }Y\text{-concave}\right\} .
\end{equation*}
\end{definition}

It is worthwhile to recall that $\mathcal{K}^{Y}\left( E,X\right) $ and $%
\mathcal{K}_{Y}\left( X,E\right) $ are vector subspaces of $\mathcal{L}%
\left( E,X\right) $ and $\mathcal{L}\left( X,E\right) ,$ respectively. Also $%
\left\Vert \cdot \right\Vert _{K^{Y}}$ and $\left\Vert \cdot \right\Vert
_{K_{Y}}$ define norms on $\mathcal{K}^{Y}\left( E,X\right) $ and $\mathcal{K%
}_{Y}\left( X,E\right) $.

\section{Vector sequence spaces}

Recall $E,F$ are Banach spaces, $X,W,Z$ are Banach lattices and $Y$ is a
Banach sequence lattice. We will treat the elements of $E^{%
\mathbb{N}
}$ indistinctly as vector sequences $\left\{ w_{j}\right\} _{j=1}^{\infty },$
or as functions $w:%
\mathbb{N}
\rightarrow E$ such that $w\left( j\right) :=w_{j}\in E,j\in 
\mathbb{N}
.$ Also, given $w=\left\{ w_{j}\right\} _{j=1}^{\infty }\in E^{%
\mathbb{N}
},$ we will denote by $\left\{ w_{j}\right\} _{j=1}^{n}$ the sequence whose
first $n$ components are $w_{1},...,w_{n}$ and $w_{j}=0$ for $j>n.$
Similarly we will denote $w\cdot \chi _{_{%
\mathbb{N}
\backslash \{1,...,n\}}}$ by $\left\{ w_{j}\right\} _{j=n+1}^{\infty }.$
Along this work, when $E$ is a Riesz space, we will consider the order by
components in $E^{%
\mathbb{N}
}$. Clearly, with such order, $E^{%
\mathbb{N}
}$ is a Riesz space.

\bigskip

For each $w=\left\{ w_{j}\right\} _{j=1}^{\infty }\in E^{%
\mathbb{N}
}$ let $J_{w}:%
\mathbb{N}
\rightarrow 
\mathbb{R}
$ be the function given by $J_{w}\left( j\right) :=\left\Vert
w_{j}\right\Vert _{E},j\in 
\mathbb{N}
.$ Next let us define the space%
\begin{equation*}
\ell ^{Y}\left( E\right) :=\left\{ w\in E^{%
\mathbb{N}
}:J_{w}\in Y\right\} =\left\{ \left\{ w_{j}\right\} _{j=1}^{\infty }\in E^{%
\mathbb{N}
}:\left\{ \left\Vert w_{j}\right\Vert _{E}\right\} _{j=1}^{\infty }\in
Y\right\}
\end{equation*}%
and the function $\left\Vert \cdot \right\Vert _{\ell ^{Y}\left( E\right)
}:\ell ^{Y}\left( E\right) \rightarrow 
\mathbb{R}
$ by%
\begin{equation*}
\left\Vert w\right\Vert _{\ell ^{Y}\left( E\right) }:=\left\Vert
J_{w}\right\Vert _{Y}=\left\Vert \left\{ \left\Vert w_{j}\right\Vert
_{E}\right\} _{j=1}^{\infty }\right\Vert _{Y},\forall w=\left\{
w_{j}\right\} _{j=1}^{\infty }\in \ell ^{Y}\left( E\right) .
\end{equation*}%
\indent Clearly $\left( \ell ^{Y}\left( E\right) ,\left\Vert \cdot
\right\Vert _{\ell ^{Y}\left( E\right) }\right) $ is a Riesz space. If in
addition $E$ is a normed lattice, $\left\Vert \cdot \right\Vert _{\ell
^{Y}\left( E\right) }$ is an order preserving norm and consequently $\ell
^{Y}\left( E\right) $ is a normed lattice. Also, given $w,z\in \ell
^{Y}\left( E\right) ,$%
\begin{eqnarray*}
\left\Vert J_{w}-J_{z}\right\Vert _{Y} &=&\left\Vert \left\{ \left\Vert
w_{j}\right\Vert _{E}\right\} _{j=1}^{\infty }-\left\{ \left\Vert
z_{j}\right\Vert _{E}\right\} _{j=1}^{\infty }\right\Vert _{Y}=\left\Vert
\left\{ \left\vert \left\Vert w_{j}\right\Vert _{E}-\left\Vert
z_{j}\right\Vert _{E}\right\vert \right\} _{j=1}^{\infty }\right\Vert _{Y} \\
&\leq &\left\Vert \left\{ \left\Vert w_{j}-z_{j}\right\Vert _{E}\right\}
_{j=1}^{\infty }\right\Vert _{Y}=\left\Vert J_{w-z}\right\Vert
_{Y}=\left\Vert w-z\right\Vert _{\ell ^{Y}\left( E\right) }.
\end{eqnarray*}%
\indent Thus, the function $J:\ell ^{Y}\left( E\right) \rightarrow Y$ given
by $J\left( w\right) :=J_{w},w\in E^{%
\mathbb{N}
}$ is a Lipschitz function.

\bigskip

We will denote by $\ell _{00}^{Y}\left( E\right) $ the linear subspace of
eventually zero sequences in $\ell ^{Y}\left( E\right) $ and we will write $%
\ell _{0}^{Y}\left( E\right) $ to denote the closure of $\ell
_{00}^{Y}\left( E\right) $ in $\ell ^{Y}\left( E\right) .$

\bigskip

Observe that, for every $n\in 
\mathbb{N}
,$ there is a natural linear isometry $E^{n}\rightarrow \ell ^{Y}\left(
E\right) $ given by $\left( w_{1},...,w_{n}\right) \rightarrow \left(
w_{1},...,w_{n},0,...\right) $. Furthermore%
\begin{equation}
\left( E,\left\Vert \cdot \right\Vert _{E^{1},Y}\right) \hookrightarrow
\left( E^{2},\left\Vert \cdot \right\Vert _{E^{2},Y}\right) \hookrightarrow
...\ell _{00}^{Y}\left( E\right) \hookrightarrow \ell _{0}^{Y}\left(
E\right) \hookrightarrow \ell ^{Y}\left( E\right) .  \label{e}
\end{equation}%
Then, from the properties of $\left( E^{n},\left\Vert \cdot \right\Vert
_{E^{n},Y}\right) $-spaces, it follows that if $\left\{ w^{m}\right\}
_{m=1}^{\infty }$ is a Cauchy (resp. convergent) sequence in $\ell
^{Y}\left( E\right) ,$ then it is a Cauchy (resp. convergent) sequence in $E$
in each one of its components.$\smallskip $\newline

\begin{theorem}
$\ell ^{Y}\left( E\right) $ and $\ell _{0}^{Y}\left( E\right) $ are Banach
spaces. If in addition $E$ is a Banach lattice, then $\ell ^{Y}\left(
E\right) $ and $\ell _{0}^{Y}\left( E\right) $ are Banach lattices.\label%
{compl2}
\end{theorem}

\begin{proof}
Let $\left\{ x^{m}\right\} _{m=1}^{\infty }\subset \ell ^{Y}\left( E\right) $
be a Cauchy sequence, where $x^{m}=\left\{ x_{j}^{m}\right\} _{j=1}^{\infty
},$\linebreak $m\in 
\mathbb{N}
.$ Since the function $J:\ell ^{Y}\left( E\right) \rightarrow Y$ is
Lipschitz, $\left\{ J_{x^{m}}\right\} _{m=1}^{\infty }$ is a Cauchy sequence
in $Y$ and thus converges to some $f\in Y.$ On the other side, for each $%
n\in 
\mathbb{N}
,$ the sequence $\left\{ x_{n}^{m}\right\} _{m=1}^{\infty }$ is convergent
in $E$ to some $x_{n}.$ Let $x:=\left\{ x_{n}\right\} _{n=1}^{\infty }\in E^{%
\mathbb{N}
}.$ Then%
\begin{equation*}
\left\vert J_{x^{m}}\left( n\right) -J_{x}\left( n\right) \right\vert
=\left\vert \left\Vert x_{n}^{m}\right\Vert _{E}-\left\Vert x_{n}\right\Vert
_{E}\right\vert \leq \left\Vert x_{n}^{m}-x_{n}\right\Vert _{E}\rightarrow
_{m\rightarrow \infty }0,\forall n\in 
\mathbb{N}
.
\end{equation*}%
\indent Therefore $\left\{ J_{x^{m}}\right\} _{m=1}^{\infty }$ converges
pointwise to $J_{x}.$ So, by Corollary \ref{coromuctp}, $J_{x}=f\in Y$ and
consequently $x\in \ell ^{Y}\left( E\right) .$\newline
\newline
\indent Since $\left\{ x^{m}-x\right\} _{m=1}^{\infty }$ is a Cauchy
sequence in $\ell ^{Y}\left( E\right) ,$ we have that $\left\{
J_{x^{m}-x}\right\} _{m=1}^{\infty }$ is a Cauchy sequence in $Y$ and thus
convergent. Also 
\begin{equation*}
J_{x^{m}-x}\left( n\right) =\left\Vert x_{n}^{m}-x_{n}\right\Vert
_{E}\rightarrow _{m\rightarrow \infty }0,\forall n\in 
\mathbb{N}
.
\end{equation*}%
\indent Thus, $\left\{ J_{x^{m}-x}\right\} _{m=1}^{\infty }$ converges
pointwise to zero. Then, by Corollary \ref{coromuctp}, $\left\Vert
J_{x^{m}-x}\right\Vert _{Y}\rightarrow 0$ and consequently $\left\Vert
x^{m}-x\right\Vert _{\ell ^{Y}\left( E\right) }\rightarrow 0.$\newline
\newline
\indent The above proves that $\ell ^{Y}\left( E\right) ,$ and consequently $%
\ell _{0}^{Y}\left( E\right) ,$ are complete. Therefore, if $E$ is a Banach
lattice, then $\ell ^{Y}\left( E\right) $ is a Banach lattice.
\end{proof}

\bigskip

Observe that. by the above theorem, the normed space $\ell _{0}^{Y}\left(
E\right) $ is complete.

\begin{lemma}
Let $E$ be a Banach space and $Y$ a Banach sequence space. Then\label{1708} 
\begin{equation}
\left\Vert \left\{ w_{j}\right\} _{j=n}^{\infty }\right\Vert _{\ell
^{Y}\left( E\right) }\downarrow 0,\forall \left\{ w_{j}\right\}
_{j=1}^{\infty }\in \ell _{0}^{Y}\left( E\right)  \label{NUEVO1}
\end{equation}%
and in consequence $\left\{ w_{j}\right\} _{j=1}^{n}\rightarrow
_{n\rightarrow \infty }\left\{ w_{j}\right\} _{j=1}^{\infty }$ in $\ell
^{Y}\left( E\right) .$
\end{lemma}

\begin{proof}
Let $\left\{ w_{j}\right\} _{j=1}^{\infty }\in \ell _{0}^{Y}\left( E\right)
. $ Clearly $\left\Vert \left\{ w_{j}\right\} _{j=n}^{\infty }\right\Vert
_{\ell ^{Y}\left( E\right) }\downarrow $. Let us fix $\epsilon >0$ and take $%
\left\{ v_{j}\right\} _{j=1}^{\infty }\in \ell _{00}^{Y}\left( E\right) $
such that 
\begin{equation*}
\left\Vert \left\{ w_{j}\right\} _{j=1}^{\infty }-\left\{ v_{j}\right\}
_{j=1}^{\infty }\right\Vert _{\ell ^{Y}\left( E\right) }\leq \epsilon .
\end{equation*}%
\indent Let $n\in 
\mathbb{N}
$ be such that $v_{j}=0$ for $j\geq n$. Then%
\begin{equation*}
\left\Vert \left\{ w_{j}\right\} _{j=n}^{\infty }\right\Vert _{\ell
^{Y}\left( E\right) }=\left\Vert \left\{ w_{j}\right\} _{j=n}^{\infty
}-\left\{ v_{j}\right\} _{j=n}^{\infty }\right\Vert _{\ell ^{Y}\left(
E\right) }\leq \left\Vert \left\{ w_{j}\right\} _{j=1}^{\infty }-\left\{
v_{j}\right\} _{j=1}^{\infty }\right\Vert _{\ell ^{Y}\left( E\right) }\leq
\epsilon .
\end{equation*}%
\indent Thus $\left( \ref{NUEVO1}\right) $ is satified and consequently $%
\left\{ w_{j}\right\} _{j=1}^{n}\rightarrow \left\{ w_{j}\right\}
_{j=1}^{\infty }$ in $\ell ^{Y}\left( E\right) .$
\end{proof}

\begin{definition}
Let $X$ be a Banach lattice. Then we define the space%
\begin{equation*}
\ell ^{\tau ,Y}\left( X\right) :=\left\{ \left\{ x_{j}\right\}
_{j=1}^{\infty }\in X^{%
\mathbb{N}
}:\sup_{n\in 
\mathbb{N}
}\left\{ \left\Vert \left( x_{1},...,x_{n}\right) \right\Vert _{X^{n},Y,\tau
}\right\} <\infty \right\}
\end{equation*}%
and the function $\left\Vert \cdot \right\Vert _{\ell ^{\tau ,Y}\left(
X\right) }:\ell ^{\tau ,Y}\left( X\right) \rightarrow 
\mathbb{R}
$ by%
\begin{equation*}
\left\Vert \left\{ x_{j}\right\} _{j=1}^{\infty }\right\Vert _{\ell ^{\tau
,Y}\left( X\right) }:=\sup_{n\in 
\mathbb{N}
}\left\{ \left\Vert \left( x_{1},...,x_{n}\right) \right\Vert _{X^{n},Y,\tau
}\right\} .
\end{equation*}
\end{definition}

\bigskip

Clearly $\left\Vert \cdot \right\Vert _{\ell ^{\tau ,Y}\left( X\right) }$ is
an order preserving norm and consequently $\ell ^{\tau ,Y}\left( X\right) $
is a normed lattice. We will denote by $\ell _{00}^{\tau ,Y}\left( X\right) $
the linear subspace of eventually zero sequences in $\ell ^{\tau ,Y}\left(
X\right) $ and we will write $\ell _{0}^{\tau ,Y}\left( X\right) $ to denote
the closure $\ell _{00}^{\tau ,Y}\left( X\right) $ in $\ell ^{\tau ,Y}\left(
X\right) .$

\bigskip

Similarly to $\ell ^{Y}\left( E\right) ,$ for every $n\in 
\mathbb{N}
,$ there exists a natural isometry $X^{n}\rightarrow \ell ^{\tau ,Y}\left(
X\right) $ given by $\left( x_{1},...,x_{n}\right) \rightarrow \left(
x_{1},...,x_{n},0,...\right) $. Furthermore%
\begin{equation}
\left( X,\left\Vert \cdot \right\Vert _{X,\tau ,Y}\right) \hookrightarrow
\left( X^{2},\left\Vert \cdot \right\Vert _{X^{2},\tau ,Y}\right)
\hookrightarrow ...\hookrightarrow \ell _{00}^{\tau ,Y}\left( X\right)
\hookrightarrow \ell _{0}^{\tau ,Y}\left( X\right)  \label{inciso}
\end{equation}%
isometrically. Hence, if $\left\{ x^{m}\right\} _{m=1}^{\infty }$ is a
Cauchy (resp. convergent) sequence in $\ell ^{\tau ,Y}\left( X\right) ,$
then it is a Cauchy (resp. convergent) sequence in $X$ in each one of its
components.$\smallskip $ Furthermore, for every $x:=\left(
x_{1},...,x_{n},0,...\right) \in \ell _{00}^{\tau ,Y}\left( X\right) $ we
have that%
\begin{equation*}
\left\Vert x\right\Vert _{\ell ^{\tau ,Y}\left( X\right) }=\left\Vert \left(
x_{1},...,x_{n}\right) \right\Vert _{X^{n},Y,\tau }.
\end{equation*}

\bigskip

The proof of the next lemma is analogous to that of Lemma \ref{1708}, we
simply replace $\left\Vert \cdot \right\Vert _{\ell ^{Y}\left( E\right) }$
by $\left\Vert \cdot \right\Vert _{\ell ^{\tau ,Y}\left( X\right) }.$

\begin{lemma}
Let $X$ be a Banach lattice. Then\label{nuevo2} 
\begin{equation*}
\left\Vert \left\{ x_{j}\right\} _{j=n}^{\infty }\right\Vert _{\ell ^{\tau
,Y}\left( X\right) }\downarrow 0,\forall \left\{ x_{j}\right\}
_{j=1}^{\infty }\in \ell _{0}^{\tau ,Y}\left( X\right)
\end{equation*}%
and in consequence $\left\{ x_{j}\right\} _{j=1}^{n}\rightarrow
_{n\rightarrow \infty }\left\{ x_{j}\right\} _{j=1}^{\infty }$ in $\ell
^{\tau ,Y}\left( X\right) .$
\end{lemma}

\begin{theorem}
Let $X$ be a Banach lattice. Then $\ell ^{\tau ,Y}\left( X\right) $ is a
Banach lattice and consequently $\ell _{0}^{\tau ,Y}\left( X\right) $ is a
Banach lattice. \label{compl}
\end{theorem}

\begin{proof}
Let $\left\{ x^{m}\right\} _{m=1}^{\infty }\subset \ell ^{\tau ,Y}\left(
X\right) $ be a Cauchy sequence, where%
\begin{equation*}
x^{m}:=\left\{ x_{j}^{m}\right\} _{j=1}^{\infty },\forall m\in 
\mathbb{N}
.
\end{equation*}%
\indent Then, for each $j\in 
\mathbb{N}
$ there exists $x_{j}\in X$ such that the sequence $\left\{
x_{j}^{m}\right\} _{m=1}^{\infty }$ converges to $x_{j}.$ Let $c\in 
\mathbb{R}
$ be such that $\left\Vert x^{m}\right\Vert _{\ell ^{\tau ,Y}\left( X\right)
}<c,m\in 
\mathbb{N}
.$ Then for each $n\in 
\mathbb{N}
,$%
\begin{equation*}
\left\Vert \left( x_{1},...,x_{n}\right) \right\Vert _{X^{n},Y,\tau
}=\lim_{m\rightarrow \infty }\left\Vert \left(
x_{1}^{m},...,x_{n}^{m}\right) \right\Vert _{X^{n},Y,\tau }\leq
\lim_{m\rightarrow \infty }\left\Vert \left\{ x^{m}\right\} _{j=1}^{\infty
}\right\Vert _{\ell ^{\tau ,Y}\left( X\right) }\leq c.
\end{equation*}%
\indent Thus $\sup_{n\in 
\mathbb{N}
}\left\{ \left\Vert \left( x_{1},...,x_{n}\right) \right\Vert _{X^{n},Y,\tau
}\right\} \leq c$ and so $x=\left\{ x_{j}\right\} _{j=1}^{\infty }\in \ell
^{\tau ,Y}\left( X\right) .$\newline
\newline
\indent Next we will prove that 
\begin{equation}
\lim_{m\rightarrow \infty }\left\Vert x-x^{m}\right\Vert _{\ell ^{\tau
,Y}\left( X\right) }=0.  \label{esd}
\end{equation}%
\newline
\indent Take $\epsilon >0$ and let $N\in 
\mathbb{N}
$ be such that 
\begin{equation}
\left\Vert x^{m}-x^{n}\right\Vert _{\ell ^{\tau ,Y}\left( X\right)
}<\epsilon ,\forall n,m\geq N.  \label{e1}
\end{equation}%
\indent Now fix $n\geq N$. Then, for each $m\geq N,$%
\begin{equation*}
\left\Vert \left( x_{1}^{m}-x_{1}^{n},...,x_{M}^{m}-x_{M}^{n}\right)
\right\Vert _{X^{M},\tau ,Y}\leq \left\Vert x^{m}-x^{n}\right\Vert _{\ell
^{\tau ,Y}\left( X\right) }<\epsilon ,\forall M\in 
\mathbb{N}
.
\end{equation*}%
\indent Thus, for each $M\in 
\mathbb{N}
$%
\begin{equation*}
\left\Vert \left( x_{1}-x_{1}^{n},...,x_{M}-x_{M}^{n}\right) \right\Vert
_{X^{M},\tau ,Y}=\lim_{m\rightarrow \infty }\left\Vert \left(
x_{1}^{m}-x_{1}^{n},...,x_{M}^{m}-x_{M}^{n}\right) \right\Vert _{X^{M},\tau
,Y}\leq \epsilon .
\end{equation*}%
\indent So%
\begin{equation*}
\left\Vert x-x^{m}\right\Vert _{\ell ^{\tau ,Y}\left( X\right) }=\sup_{n\in 
\mathbb{N}
}\left\{ \left\Vert \left( x_{1}-x_{1}^{m},...,x_{n}-x_{n}^{m}\right)
\right\Vert _{X^{n},Y,\tau }\right\} \leq \epsilon .
\end{equation*}%
\indent Then $\left( \ref{esd}\right) $ is satisfied and consequently $\ell
^{\tau ,Y}\left( X\right) $ is a Banach lattice. \newline
\newline
\indent Now we will prove that $\ell _{0}^{\tau ,Y}\left( X\right) $ is an
ideal of $\ell ^{\tau ,Y}\left( X\right) .$ Let $\left\{ x_{j}\right\}
_{j=1}^{\infty }$,$\left\{ z_{j}\right\} _{j=1}^{\infty }\in \ell _{0}^{\tau
,Y}\left( X\right) .$ Then, by the above lemma, 
\begin{equation*}
\left\Vert \left\{ x_{j}\vee z_{j}\right\} _{j=1}^{\infty }-\left\{
x_{j}\vee z_{j}\right\} _{j=1}^{n}\right\Vert _{\ell ^{\tau ,Y}\left(
X\right) }=\left\Vert \left\{ x_{j}\vee z_{j}\right\} _{j=n+1}^{\infty
}\right\Vert _{\ell ^{\tau ,Y}\left( X\right) }
\end{equation*}%
and%
\begin{eqnarray*}
\left\Vert \left\{ x_{j}\vee z_{j}\right\} _{j=n+1}^{\infty }\right\Vert
_{\ell ^{\tau ,Y}\left( X\right) }\hspace{-0.3cm} &\leq &\hspace{-0.15cm}%
\left\Vert \left\{ \left\vert x_{j}\right\vert +\left\vert z_{j}\right\vert
\right\} _{j=n+1}^{\infty }\right\Vert _{\ell ^{\tau ,Y}\left( X\right) } \\
&& \\
&\leq &\hspace{-0.15cm}\left\Vert \left\{ \left\vert x_{j}\right\vert
\right\} _{j=n+1}^{\infty }\right\Vert _{\ell ^{\tau ,Y}\left( X\right)
}+\left\Vert \left\{ \left\vert z_{j}\right\vert \right\} _{j=n+1}^{\infty
}\right\Vert _{\ell ^{\tau ,Y}\left( X\right) }\rightarrow 0.
\end{eqnarray*}%
\indent Therefore $\left\{ x_{j}\right\} _{j=1}^{\infty }\vee \left\{
z_{j}\right\} _{j=1}^{\infty }=\left\{ x_{j}\vee z_{j}\right\}
_{j=1}^{\infty }\in \ell _{0}^{\tau ,Y}\left( X\right) $ and consequently $%
\ell _{0}^{\tau ,Y}\left( X\right) $ is a Banach lattice.
\end{proof}

\bigskip

\begin{definition}
Let $V,W$ be linear spaces and $T:V\rightarrow W$ a linear operator. For
each $n\in 
\mathbb{N}
$ we define the operator $T_{n}:V^{n}\rightarrow W^{n}$ by%
\begin{equation*}
T_{n}\left( v_{1},...,v_{n}\right) :=\left( Tv_{1},...,Tv_{n}\right)
,\forall v_{1},...,v_{n}\in V.
\end{equation*}%
\indent We also define the \textit{associated operator of }$T,$ $\widetilde{T%
}:V^{%
\mathbb{N}
}\rightarrow W^{%
\mathbb{N}
}$ by%
\begin{equation}
\widetilde{T}\left( \left\{ v_{j}\right\} _{j=1}^{\infty }\right) =\left\{
Tv_{j}\right\} _{j=1}^{\infty },\forall \left\{ v\right\} _{j=1}^{\infty
}\in E^{%
\mathbb{N}
}  \label{2606}
\end{equation}%
\indent Since $T$ is a linear operator it follows that $T_{n}$ and $%
\widetilde{T}$ are linear operators and $\widetilde{T}\left( V_{00}\right)
\subset W_{00}.$
\end{definition}

\begin{lemma}
Let $X$ be a Banach lattice, $E$ a Banach space, $T:E\rightarrow X$ a
continuous linear operator. Then, for each $n\in 
\mathbb{N}
$ the operator $T_{n}:\left( E^{n},\left\Vert \cdot \right\Vert
_{E^{n},Y}\right) \rightarrow \left( X^{n},\left\Vert \cdot \right\Vert
_{X^{n},Y,\tau }\right) $ is continuous.\label{nnlema}
\end{lemma}

\begin{proof}
Fix $n\in 
\mathbb{N}
$ and let $w_{1},...,w_{n}\in E.$ Then%
\begin{eqnarray*}
\left\Vert T_{n}\left( w_{1},...,w_{n}\right) \right\Vert _{X^{n},Y}\hspace{%
-0.25cm} &=&\hspace{-0.2cm}\left\Vert \left( Tw_{1},...,Tw_{n}\right)
\right\Vert _{X^{n},Y}=\left\Vert \left( \left\Vert Tw_{1}\right\Vert
_{X},...,\left\Vert Tw_{n}\right\Vert _{X}\right) \right\Vert _{Y} \\
&& \\
&\leq &\hspace{-0.2cm}\left\Vert T\right\Vert \left\Vert \left( \left\Vert
w_{1}\right\Vert _{E},...,\left\Vert w_{n}\right\Vert _{E}\right)
\right\Vert _{Y}=\left\Vert T\right\Vert \left\Vert \left(
w_{1},...,w_{n}\right) \right\Vert _{E^{n},Y}.
\end{eqnarray*}%
\indent Since the norms $\left\Vert \cdot \right\Vert _{X^{n},Y,\tau }$ and $%
\left\Vert \cdot \right\Vert _{X^{n},Y}$ are continuous it follows that $%
T_{n}$ is a continuous operator.
\end{proof}

\begin{proposition}
Let $X$ be a Banach lattice and $E$ a Banach space. Then$\label{propo190}$%
\newline
\indent$i)$ A linear operator $T:E\rightarrow X$ is\textit{\ }$Y$-convex if
and only if $\widetilde{T}:\ell _{00}^{Y}\left( E\right) \rightarrow \ell
_{00}^{\tau ,Y}\left( X\right) $ is bounded, in which case $\left\Vert 
\widetilde{T}\right\Vert =\left\Vert T\right\Vert _{\mathcal{K}%
^{Y}}.\smallskip $\newline
\indent$ii)$ A linear operator $S:X\rightarrow E$ is $Y$-concave\ if and
only if $\widetilde{S}:\ell _{00}^{\tau ,Y}\left( X\right) \rightarrow \ell
_{00}^{Y}\left( E\right) $ is bounded, in which case $\left\Vert \widetilde{S%
}\right\Vert =\left\Vert S\right\Vert _{\mathcal{K}_{Y}}.$\newline
\end{proposition}

\begin{proof}
Let $T:E\rightarrow X$ be a\textit{\ }$Y$-convex operator. Take $\left\{
w_{j}\right\} _{j=1}^{\infty }\in \ell _{00}^{Y}\left( E\right) $ and let $%
n\in 
\mathbb{N}
$ be such that $w_{j}=0,j>n.$ Then%
\begin{eqnarray*}
\left\Vert \widetilde{T}\left( \left\{ w_{j}\right\} _{j=1}^{\infty }\right)
\right\Vert _{\ell ^{\tau ,Y}\left( X\right) }\hspace{-0.2cm} &=&\left\Vert
\left\{ Tw_{j}\right\} _{j=1}^{\infty }\right\Vert _{\ell ^{\tau ,Y}\left(
X\right) }=\left\Vert \left( Tw_{1},...,Tw_{n}\right) \right\Vert
_{X^{n},h_{n,Y}} \\
&\leq &\left\Vert T\right\Vert _{\mathcal{K}^{Y}}\left\Vert \left(
w_{1},...w_{n}\right) \right\Vert _{E^{n},Y}=\left\Vert T\right\Vert _{%
\mathcal{K}^{Y}}\left\Vert \left\{ w_{j}\right\} _{j=1}^{\infty }\right\Vert
_{\ell ^{Y}\left( E\right) }.
\end{eqnarray*}%
\indent Thus $\widetilde{T}:\ell _{00}^{Y}\left( E\right) \rightarrow \ell
_{00}^{\tau ,Y}\left( X\right) $ is bounded and 
\begin{equation}
\left\Vert \widetilde{T}\right\Vert \leq \left\Vert T\right\Vert _{\mathcal{K%
}^{Y}}.  \label{de2}
\end{equation}%
\indent On the other hand let $T:E\rightarrow X$ be a linear operator with
bounded associated operator $\widetilde{T}:\ell _{00}^{Y}\left( E\right)
\rightarrow \ell _{00}^{\tau ,Y}\left( X\right) $. Then, for every $n\in 
\mathbb{N}
$ and $w_{1},...,w_{n}\in E$%
\begin{eqnarray*}
\left\Vert \left( Tw_{1},...,Tw_{n}\right) \right\Vert _{X^{n},Y,\tau }
&=&\left\Vert \left\{ Tw_{j}\right\} _{j=1}^{\infty }\right\Vert _{\ell
^{\tau ,Y}\left( X\right) }=\left\Vert \widetilde{T}\left( \left\{
w_{j}\right\} _{j=1}^{\infty }\right) \right\Vert _{\ell ^{\tau ,Y}\left(
X\right) } \\
&\leq &\left\Vert \widetilde{T}\right\Vert \left\Vert \left\{ w_{j}\right\}
_{j=1}^{\infty }\right\Vert _{\ell ^{Y}\left( E\right) }=\left\Vert 
\widetilde{T}\right\Vert \left\Vert \left( w_{1},...w_{n}\right) \right\Vert
_{E^{n},Y}.
\end{eqnarray*}%
\indent Thus $T$ is $Y$-convex and 
\begin{equation}
\left\Vert \widetilde{T}\right\Vert \geq \left\Vert T\right\Vert _{\mathcal{K%
}^{Y}}.\newline
\label{de1}
\end{equation}%
\newline
\indent It follows by inequalities $\left( \ref{de1}\right) $ and $\left( %
\ref{de2}\right) $ that $\left\Vert \widetilde{T}\right\Vert =\left\Vert
T\right\Vert _{\mathcal{K}^{Y}}.$\newline
\newline
\indent$ii)$ This proof is analogous to $i)$.
\end{proof}

\bigskip

Next, let $T\in \mathcal{K}^{Y}\left( E,X\right) .$ Since $\ell _{0}^{\tau
,Y}\left( X\right) =\overline{\ell _{00}^{\tau ,Y}\left( X\right) }$ is
complete, the associated operator $\widetilde{T}:\ell _{00}^{Y}\left(
E\right) \rightarrow \ell _{00}^{\tau ,Y}\left( X\right) $ can be linearly
extended to $\ \overline{\ell _{00}^{Y}\left( E\right) }=\ell _{0}^{Y}\left(
E\right) .$ Let us denote such extension by $R:\ell _{0}^{Y}\left( E\right)
\rightarrow _{0}^{\tau ,Y}\left( X\right) .$ Let $w=\left\{ w_{j}\right\}
_{j=1}^{\infty }\in \ell _{0}^{Y}\left( E\right) $ and take $\left\{
w^{m}\right\} _{m=1}^{\infty }\subset \ell _{00}^{Y}\left( E\right) $ such
that $w^{m}\rightarrow w,$where%
\begin{equation*}
w^{m}:=\left\{ w_{j}^{m}\right\} _{j=1}^{\infty }\in \ell _{00}^{Y}\left(
E\right) ,\forall m\in 
\mathbb{N}
.
\end{equation*}%
\indent Observe that 
\begin{equation*}
Rw=\lim_{m\rightarrow \infty }Rw^{m}=\lim_{m\rightarrow \infty }\left\{
Tw_{j}^{m}\right\} _{j=1}^{\infty }
\end{equation*}%
and define $z:=Rw.$ Since the convergence in $\ell _{0}^{\tau ,Y}\left(
X\right) $ imply the convergence in each component we have that 
\begin{equation*}
\left\Vert Tw_{j}^{m}-z_{j}\right\Vert _{X}\rightarrow 0,\forall j\in 
\mathbb{N}
.
\end{equation*}%
\indent Then $z_{j}=\lim_{m\rightarrow \infty }Tw_{j}^{m}=T\left(
\lim_{m\rightarrow \infty }w_{j}^{m}\right) =Tw_{j}$ and consequently%
\begin{equation*}
R\left\{ w_{j}\right\} _{j=1}^{\infty }=\left\{ Tw_{j}\right\}
_{j=1}^{\infty },\forall \left\{ w_{j}\right\} _{j=1}^{\infty }\in \ell
_{0}^{Y}\left( E\right) .
\end{equation*}%
\indent Therefore we conclude that the linear extension of $\widetilde{T}$
can be expressed in the same way as $\left( \ref{2606}\right) .$
Analogously, for a $Y$-concave operatror $S:X\rightarrow E$, the associated
operator $\widetilde{S}:\ell _{00}^{\tau ,Y}\left( X\right) \rightarrow \ell
_{0}^{Y}\left( E\right) $ can be linearly extended to $\ell _{0}^{\tau
,Y}\left( X\right) .$ Thus, the following result holds:

\begin{theorem}
Let $X$ be a Banach lattice and $E$ a Banach space. Then\label{propo1902}%
\newline
\indent$i)$ A linear operator $T:E\rightarrow X$ is\textit{\ }$Y$-convex if
and only if $\widetilde{T}:\ell _{0}^{Y}\left( E\right) \rightarrow \ell
_{0}^{\tau ,Y}\left( X\right) $ is bounded, in which case $\left\Vert 
\widetilde{T}\right\Vert =\left\Vert T\right\Vert _{\mathcal{K}%
^{Y}}.\smallskip $\newline
\indent$ii)$ A linear operator $S:X\rightarrow E$ is $Y$-concave\ if and
only if $\widetilde{S}:\ell _{0}^{\tau ,Y}\left( X\right) \rightarrow \ell
_{0}^{Y}\left( E\right) $ is bounded, in which case $\left\Vert \widetilde{S}%
\right\Vert =\left\Vert S\right\Vert _{\mathcal{K}_{Y}}.$
\end{theorem}

\begin{proposition}
Let $X$ be a Banach lattice, $E$ a Banach space and $T:E\rightarrow X$ a
linear operator. If $T\in \mathcal{L}\left( E,X\right) $ and $\widetilde{T}%
\left( \ell _{0}^{Y}\left( E\right) \right) \subset \ell _{0}^{\tau
,Y}\left( X\right) ,$ then $\widetilde{T}:\ell _{0}^{Y}\left( E\right)
\rightarrow \ell _{0}^{\tau ,Y}\left( X\right) $ is bounded. In consequence $%
T\in \mathcal{K}^{Y}\left( E,X\right) .$
\end{proposition}

\begin{proof}
Let $\left\{ w_{n}\right\} _{n=1}^{\infty }\subset \ell _{0}^{Y}\left(
E\right) $ and $s=\left\{ s_{j}\right\} _{j=1}^{\infty }\in \ell
_{0}^{Y}\left( E\right) $ such that $\left\Vert w_{n}-s\right\Vert _{\ell
^{Y}\left( E\right) }$\linebreak $\rightarrow 0,$ where $w_{n}=\left\{
w_{n,j}\right\} _{j=1}^{\infty },n\in 
\mathbb{N}
$. Assume that $\left\{ \widetilde{T}\left( w_{n}\right) \right\}
_{n=1}^{\infty }$ converges to $y=\left\{ y_{j}\right\} _{j=1}^{\infty }\in
\ell ^{\tau ,Y}\left( X\right) .$ Now we will prove that $y=\widetilde{T}%
\left( s\right) $ and consequently, by the closed graph theorem, that $%
\widetilde{T}:\ell _{0}^{Y}\left( E\right) \rightarrow \ell _{0}^{\tau
,Y}\left( X\right) $ is bounded. \newline
\newline
\indent Since $\left\{ \widetilde{T}\left( w_{n}\right) \right\}
_{n=1}^{\infty }$ and $\left\{ w_{n}\right\} _{n=1}^{\infty }$ converges in
each component to $y$ and $s,$ respectively, we have that%
\begin{equation}
\left\Vert T\left( w_{n,j}\right) -y_{j}\right\Vert _{X}\rightarrow
_{n\rightarrow \infty }0,\forall j\in 
\mathbb{N}
\label{final}
\end{equation}%
and%
\begin{equation*}
\left\Vert w_{n,j}-s_{j}\right\Vert _{E}\rightarrow _{n\rightarrow \infty
}0,\forall j\in 
\mathbb{N}
.
\end{equation*}%
\indent Then%
\begin{equation}
\left\Vert T\left( w_{n,j}\right) -T\left( s_{j}\right) \right\Vert
_{X}=\left\Vert T\left( w_{n,j}-s_{j}\right) \right\Vert _{X}\leq \left\Vert
T\right\Vert \left\Vert w_{n,j}-s_{j}\right\Vert _{E}\rightarrow
_{n\rightarrow \infty }0,\forall j\in 
\mathbb{N}
.  \label{final2}
\end{equation}%
\indent So, by $\left( \ref{final}\right) $ and $\left( \ref{final2}\right) $
it follows that $y_{j}=T\left( s_{j}\right) ,j\in 
\mathbb{N}
.$ Thus $y=\widetilde{T}\left( s\right) $ and consequently $\widetilde{T}%
:\ell _{0}^{Y}\left( E\right) \rightarrow \ell _{0}^{\tau ,Y}\left( X\right) 
$ is bounded. Finally by the above theorem $T\in \mathcal{K}^{Y}\left(
E,X\right) .$
\end{proof}

\bigskip

In a similar way, if $S\in \mathcal{L}\left( X,E\right) $ and $\widetilde{S}%
\left( \ell _{0}^{\tau ,Y}\left( X\right) \right) \subset \ell
_{0}^{Y}\left( E\right) ,$ then $\widetilde{S}:\ell _{0}^{\tau ,Y}\left(
X\right) \rightarrow \ell _{0}^{Y}\left( E\right) $ is bounded and $S\in 
\mathcal{K}_{Y}\left( X,E\right) .$

\section{The dual space of $\ell _{0}^{Y}\left( E\right) $}

Recall $E,F$ are Banach spaces, $X,W,Z$ are Banach lattices and $Y$ is a
Banach sequence lattice.

\begin{definition}
Let $\left( \Omega ,\Sigma ,\mu \right) $ be a measure space. A $\mu $%
-B.f.s. is said to be \textit{saturated} if, for every $A\in \Sigma $ of
positive measure$,$ there exists a measurable set $B\subset A$ such that $%
\mu \left( B\right) >0$ and $\chi _{_{B}}\in V.$
\end{definition}

Remember that given a measure space $\left( \Omega ,\Sigma ,\mu \right) $
and a saturated $\mu $-B.f.s. $V\subset L^{0}(\mu )$ the \textit{associated K%
\"{o}the space }$V^{\times }$ is defined by 
\begin{equation*}
V^{\times }:=\left\{ g\in L^{0}(\mu ):g\cdot f\in L^{1}(\mu ),\forall f\in
V\right\}
\end{equation*}%
\indent Clearly $V^{\times }$ is a linear space and, since $V$ is saturated,
the function $\left\Vert \cdot \right\Vert _{V^{\times }}:V^{\times
}\rightarrow 
\mathbb{R}
$ given by%
\begin{equation*}
\text{ }\left\Vert g\right\Vert _{V^{\times }}:=\sup \left\{ \int_{\Omega
}\left\vert gf\right\vert d\mu :f\in B_{V}\right\} <\infty ,
\end{equation*}%
is a norm on $V^{\times }$.\newline
\newline
\indent Observe that a Banach sequence lattice $Y$ is a saturated $\mu $%
-B.f.s.. In this case we have that%
\begin{equation*}
Y^{\times }:=\left\{ \left\{ \beta _{n}\right\} _{n=1}^{\infty }\in 
\mathbb{R}
^{%
\mathbb{N}
}:\sum_{n=1}^{\infty }\left\vert \alpha _{n}\beta _{n}\right\vert <\infty
,\forall \left\{ \alpha _{n}\right\} _{n=1}^{\infty }\in Y\right\}
\end{equation*}%
and, for each $\left\{ \beta _{n}\right\} _{n=1}^{\infty }\in Y^{\times },$ 
\begin{equation*}
\left\Vert \left\{ \beta _{n}\right\} _{n=1}^{\infty }\right\Vert
_{Y^{\times }}:=\sup \left\{ \sum_{n=1}^{\infty }\left\vert \alpha _{n}\beta
_{n}\right\vert :\left\{ \alpha _{n}\right\} _{n=1}^{\infty }\in
B_{Y}\right\} <\infty .
\end{equation*}

Recall a Banach lattice $X$ has the $\sigma $\textit{-Fatou property }if
each increasing sequence $\left\{ f_{n}\right\} _{n=1}^{\infty }\subset
X^{+} $ such that $\sup_{n\in 
\mathbb{N}
}\left\Vert f_{n}\right\Vert _{X}<\infty ,$ satisfies $f=\sup_{n\in 
\mathbb{N}
}f_{n}\in X$ and $\left\Vert f\right\Vert _{X}=\sup_{n\in 
\mathbb{N}
}\left\Vert f_{n}\right\Vert _{X}.$

\begin{proposition}
Let $Y$ be a Banach sequence lattice. Then $Y^{\times }$ is a Banach
sequence lattice with the $\sigma $-Fatou property.\label{propo1810}
\end{proposition}

\begin{proposition}
$\left( \text{H\"{o}lder%
\'{}%
s inequality}\right) $ Let $Y$ be a Banach sequence lattice. Then, for every 
$\left\{ \alpha _{n}\right\} _{n=1}^{\infty }\in Y$ and $\left\{ \beta
_{n}\right\} _{n=1}^{\infty }\in Y^{\times }$\label{holder}%
\begin{equation}
\sum_{n=1}^{\infty }\left\vert \alpha _{n}\beta _{n}\right\vert \leq
\left\Vert \left\{ \alpha _{n}\right\} _{n=1}^{\infty }\right\Vert _{%
\mathbb{R}
^{n},Y}\left\Vert \left\{ \beta _{n}\right\} _{n=1}^{\infty }\right\Vert _{%
\mathbb{R}
^{n},Y^{\times }}.  \label{500}
\end{equation}
\end{proposition}

\bigskip

Next, given a Banach space $E,$ we will represent the dual space of $\ell
_{0}^{Y}\left( E\right) .$ Fix $s=\left\{ \varphi _{n}\right\}
_{n=1}^{\infty }\in \ell ^{Y^{\times }}\left( E^{\ast }\right) $ and define%
\begin{equation}
\rho _{s}\left( \left\{ w_{n}\right\} _{n=1}^{\infty }\right)
:=\sum_{n=1}^{\infty }\varphi _{n}\left( w_{n}\right) ,\forall \left\{
w_{n}\right\} _{n=1}^{\infty }\in \ell _{0}^{Y}\left( E\right) .
\label{dedede}
\end{equation}%
\indent From H\"{o}lder%
\'{}%
s inequality%
\begin{eqnarray}
\sum_{n=1}^{\infty }\left\vert \varphi _{n}\left( w_{n}\right) \right\vert
&\leq &\sum_{n=1}^{\infty }\left\Vert \varphi _{n}\right\Vert _{E^{\ast
}}\left\Vert w_{n}\right\Vert _{E}  \label{wrt} \\
&&  \notag \\
&\leq &\left\Vert \left\{ \left\Vert \varphi _{n}\right\Vert _{E^{\ast
}}\right\} _{n=1}^{\infty }\right\Vert _{Y^{\times }}\left\Vert \left\{
\left\Vert w_{n}\right\Vert _{E}\right\} _{n=1}^{\infty }\right\Vert _{Y} 
\notag \\
&&  \notag \\
&=&\left\Vert \left\{ \varphi _{n}\right\} _{n=1}^{\infty }\right\Vert
_{\ell ^{Y^{\times }}\left( E^{\ast }\right) }\left\Vert \left\{
w_{n}\right\} _{n=1}^{\infty }\right\Vert _{\ell ^{Y}\left( E\right) } 
\notag \\
&&  \notag \\
&=&\left\Vert s\right\Vert _{\ell ^{Y^{\times }}\left( E^{\ast }\right)
}\left\Vert \left\{ w_{n}\right\} _{n=1}^{\infty }\right\Vert _{\ell
^{Y}\left( E\right) }<\infty .  \notag
\end{eqnarray}%
\indent Thus the functional $\rho _{s}:\ell _{0}^{Y}\left( E\right)
\rightarrow 
\mathbb{R}
$ is well defined for each $s\in \ell ^{Y^{\times }}\left( E^{\ast }\right) $
and it is continuous. \newline
\newline
\indent Now let us define the operator $\rho :\ell ^{Y^{\times }}\left(
E^{\ast }\right) \rightarrow \ell _{0}^{Y}\left( E\right) ^{\ast }$ by%
\begin{equation}
\rho \left( s\right) :=\rho _{s},\forall s\in \ell ^{Y^{\times }}\left(
E^{\ast }\right) .  \label{line}
\end{equation}%
\indent Observe that $\rho $ is a linear operator and, by $\left( \ref{wrt}%
\right) ,$ 
\begin{equation}
\left\Vert \rho _{s}\right\Vert _{\ell ^{Y}\left( E\right) ^{\ast }}\leq
\left\Vert s\right\Vert _{\ell ^{Y^{\times }}\left( E^{\ast }\right)
},\forall s\in \ell ^{Y^{\times }}\left( E^{\ast }\right) .  \label{iso1}
\end{equation}%
\indent Therefore $\rho $ is bounded and $\left\Vert \rho \right\Vert \leq
1. $

\begin{theorem}
Let $E$ be a Banach space and $\rho :\ell ^{Y^{\times }}\left( E^{\ast
}\right) \rightarrow \ell _{0}^{Y}\left( E\right) ^{\ast }$ the function
defined in $\left( \ref{line}\right) .$ Then $\rho $ is an isometric
isomorphism. We will denote this by\label{tmapen} 
\begin{equation*}
\ell _{0}^{Y}\left( E\right) ^{\ast }=\ell ^{Y^{\times }}\left( E^{\ast
}\right) .
\end{equation*}
\end{theorem}

\begin{proof}
Let $\varphi \in \ell _{0}^{Y}\left( E\right) ^{\ast }.$ For each $n\in 
\mathbb{N}
$ define 
\begin{equation*}
\varphi _{n}\left( w\right) :=\varphi \left( e_{n}\cdot w\right) ,\forall
w\in E.
\end{equation*}%
\indent Observe that $\varphi _{n}$ is a linear operator and%
\begin{equation*}
\left\vert \varphi _{n}\left( w\right) \right\vert \leq \left\Vert \varphi
\right\Vert \left\Vert e_{n}\cdot w\right\Vert _{\ell ^{Y}\left( E\right)
}=\left\Vert \varphi \right\Vert \left\Vert e_{n}\right\Vert _{Y}\left\Vert
w\right\Vert _{E},\forall w\in E.
\end{equation*}%
\indent Thus $\varphi _{n}\in E^{\ast }.$ \newline
\newline
\indent Next we will prove that $\left\{ \varphi _{n}\right\} _{n=1}^{\infty
}\in \ell ^{Y^{\times }}\left( E^{\ast }\right) $ i.e., that $\left\{
\left\Vert \varphi _{n}\right\Vert _{E^{\ast }}\right\} _{n=1}^{\infty }\in
Y^{\times }.$ Recall that\ $\left\{ \left\Vert \varphi _{n}\right\Vert
_{E^{\ast }}\right\} _{n=1}^{\infty }\in Y^{\times }$ if 
\begin{equation*}
\sum_{n=1}^{\infty }\left\vert \lambda _{n}\right\vert \left\Vert \varphi
_{n}\right\Vert _{E^{\ast }}<\infty ,\forall \left\{ \lambda _{n}\right\}
_{n=1}^{\infty }\in Y.
\end{equation*}%
\indent Fix $m\in 
\mathbb{N}
$ and let $\left( \left\Vert \varphi _{1}\right\Vert _{E^{\ast
}},...,\left\Vert \varphi _{m}\right\Vert _{E^{\ast }}\right) \in 
\mathbb{R}
^{m}.$ Since $\left( 
\mathbb{R}
^{m},\left\Vert \cdot \right\Vert _{Y}\right) $ is reflexive there exists $%
\left( \lambda _{1},...,\lambda _{m}\right) \in 
\mathbb{R}
^{m}$ such that $\left\Vert \left( \lambda _{1},...,\lambda _{m}\right)
\right\Vert _{Y}=1$ and%
\begin{equation*}
\left\Vert \left\{ \left\Vert \varphi _{n}\right\Vert _{E^{\ast }}\right\}
_{n=1}^{m}\right\Vert _{Y^{\times }}=\left\Vert \left( \left\Vert \varphi
_{1}\right\Vert _{E^{\ast }},...,\left\Vert \varphi _{m}\right\Vert
_{E^{\ast }}\right) \right\Vert _{Y^{\ast }}=\sum_{n=1}^{m}\left\Vert
\varphi _{n}\right\Vert _{E^{\ast }}\lambda _{n}.
\end{equation*}%
\indent Fix $\epsilon >0.$ For each $1\leq n\leq m$ take $z_{n}\in B_{E}$
such that $\left\Vert \varphi _{n}\right\Vert _{E^{\ast }}\leq \varphi
_{n}\left( z_{n}\right) +\frac{\epsilon }{\left\vert \lambda _{n}\right\vert
m+1}.$ Then%
\begin{eqnarray}
\left\Vert \left\{ \left\Vert \varphi _{n}\right\Vert _{E^{\ast }}\right\}
_{n=1}^{m}\right\Vert _{Y^{\times }}\hspace{-6pt} &\leq
&\sum_{n=1}^{m}\left\Vert \varphi _{n}\right\Vert _{E^{\ast }}\left\vert
\lambda _{n}\right\vert  \label{aq} \\
&&  \notag \\
&\leq &\sum_{n=1}^{m}\left( \varphi _{n}\left( z_{n}\right) +\frac{\epsilon 
}{\left\vert \lambda _{n}\right\vert m+1}\right) \left\vert \lambda
_{n}\right\vert  \notag \\
&&  \notag \\
&\leq &\sum_{n=1}^{m}\left( \varphi _{n}\left( \left\vert \lambda
_{n}\right\vert z_{n}\right) +\frac{\epsilon }{m}\right) =\varphi \left(
\left\{ \left\vert \lambda _{n}\right\vert z_{n}\right\} _{n=1}^{m}\right)
+\epsilon  \notag \\
&&  \notag \\
&\leq &\left\Vert \varphi \right\Vert _{\ell ^{Y}\left( E\right) ^{\ast
}}\left\Vert \left\{ \left\vert \lambda _{n}\right\vert z_{n}\right\}
_{n=1}^{m}\right\Vert _{\ell ^{Y}\left( E\right) }+\epsilon  \notag \\
&&  \notag \\
&=&\left\Vert \varphi \right\Vert _{\ell ^{Y}\left( E\right) ^{\ast
}}\left\Vert \left( \left\Vert \left\vert \lambda _{1}\right\vert
z_{1}\right\Vert _{E}...,\left\Vert \left\vert \lambda _{m}\right\vert
z_{m}\right\Vert _{E}\right) \right\Vert _{Y}+\epsilon  \notag \\
&&  \notag \\
&\leq &\left\Vert \varphi \right\Vert _{\ell ^{Y}\left( E\right) ^{\ast
}}\left\Vert \left( \left\vert \lambda _{1}\right\vert ,...,\left\vert
\lambda _{m}\right\vert \right) \right\Vert _{Y}+\epsilon  \notag \\
&&  \notag \\
&=&\left\Vert \varphi \right\Vert _{\ell ^{Y}\left( E\right) ^{\ast
}}\left\Vert \left( \lambda _{1},...,\lambda _{m}\right) \right\Vert
_{Y}+\epsilon =\left\Vert \varphi \right\Vert _{\ell ^{Y}\left( E\right)
^{\ast }}+\epsilon .  \notag
\end{eqnarray}%
\indent Therefore%
\begin{equation}
\left\Vert \left\{ \left\Vert \varphi _{n}\right\Vert _{E^{\ast }}\right\}
_{n=1}^{m}\right\Vert _{Y^{\times }}\leq \left\Vert \varphi \right\Vert
_{\ell ^{Y}\left( E\right) ^{\ast }}.  \label{nal}
\end{equation}%
\indent Now, as seen in Proposition \ref{propo1810}, $Y^{\times }$ is $%
\sigma $-Fatou. Thus $\left\{ \left\Vert \varphi _{n}\right\Vert _{E^{\ast
}}\right\} _{n=1}^{\infty }=\sup_{m\in 
\mathbb{N}
}\left\{ \left\Vert \varphi _{n}\right\Vert _{E^{\ast }}\right\}
_{n=1}^{m}\in Y^{\times }$ and 
\begin{equation}
\left\Vert \left\{ \left\Vert \varphi _{n}\right\Vert _{E^{\ast }}\right\}
_{n=1}^{\infty }\right\Vert _{Y^{\times }}=\sup_{m\in 
\mathbb{N}
}\left\Vert \left\{ \left\Vert \varphi _{n}\right\Vert _{E^{\ast }}\right\}
_{n=1}^{m}\right\Vert _{Y^{\times }}\leq \left\Vert \varphi \right\Vert
_{\ell ^{Y}\left( E\right) ^{\ast }}.  \label{iso4}
\end{equation}%
That is, $\left\{ \varphi _{n}\right\} _{n=1}^{\infty }\in \ell ^{Y^{\times
}}\left( E^{\ast }\right) $ and 
\begin{equation*}
\left\Vert \left\{ \varphi _{n}\right\} _{n=1}^{\infty }\right\Vert _{\ell
^{Y^{\times }}\left( E^{\ast }\right) }\leq \left\Vert \varphi \right\Vert
_{\ell ^{Y}\left( E\right) ^{\ast }}.
\end{equation*}%
\indent Finally we will prove that $\rho \left( \left\{ \varphi _{n}\right\}
_{n=1}^{\infty }\right) =\varphi .$ Let~$\left\{ w_{n}\right\}
_{n=1}^{\infty }\in \ell _{0}^{Y}\left( E\right) .$ Then, by the continuity
of $\varphi $ and Lemma \ref{1708}, 
\begin{eqnarray}
\rho \left( \left\{ \varphi _{n}\right\} _{n=1}^{\infty }\right) \left(
\left\{ w_{n}\right\} _{n=1}^{\infty }\right) &=&\sum_{n=1}^{\infty }\varphi
_{n}\left( w_{n}\right) =\lim_{m\rightarrow \infty }\sum_{n=1}^{m}\varphi
_{n}\left( w_{n}\right)  \label{ewq} \\
&&  \notag \\
&=&\lim_{m\rightarrow \infty }\varphi \left( \left\{ w_{n}\right\}
_{n=1}^{m}\right) =\varphi \left( \left\{ w_{n}\right\} _{n=1}^{\infty
}\right) .  \notag
\end{eqnarray}%
\indent Thus $\rho \left( \left\{ \varphi _{n}\right\} _{n=1}^{\infty
}\right) =\varphi $ and consequently $\rho $ is bijective. From $\left( \ref%
{iso1}\right) $ and $\left( \ref{iso4}\right) \ $we conclude that $\rho $ is
an isometry.
\end{proof}

\section{The dual space of $\ell _{0}^{\protect\tau ,Y}\left( X\right) $}

Now, similarly to the previous section, we analyze the dual space of $\ell
_{0}^{\tau ,Y}\left( X\right) $ and prove that it is a also a vector
sequence space. In order to do this we will begin by recalling some
properties of abstract $L^{1}$-spaces.\newline

\bigskip

A Banach lattice $X_{1}$ is an \textit{abstract }$L^{1}$\textit{-space if}%
\begin{equation}
\left\Vert x+y\right\Vert _{X_{1}}=\left\Vert x\right\Vert
_{X_{1}}+\left\Vert y\right\Vert _{X_{1}},\forall x,y\in X_{1}^{+}\text{
s.t.. }x\wedge y=0.  \label{propnue}
\end{equation}

Observe that, if $\left\Vert \cdot \right\Vert _{X_{1}}$ is additive in the
positive cone, then $X_{1}$ is an abstract $L^{1}$-space.

\bigskip

It is well-known that, given an abstract $L^{1}$-space $X_{1},$ there exists
a localizable measure space $\left( \Omega ,\Sigma ,\mu \right) $ and an
isometric lattice isomorphism between $X_{1}$ and $L^{1}\left( \mu \right) $ 
\cite[Theorem 1.b.2]{Linden}. Furthermore, in \cite[Vol. II, Theorem 243G]%
{Fremlin} is proven that the operator $R:L^{\infty }\left( \mu \right)
\rightarrow L^{1}\left( \mu \right) ^{\ast }$ given by%
\begin{equation}
\left\langle f,Rg\right\rangle =\int_{\Omega }f\left( w\right) \cdot g\left(
w\right) dw,\forall g\in L^{\infty }\left( \mu \right) ,\forall f\in
L^{1}\left( \mu \right) .  \label{juju2}
\end{equation}%
is an isometric lattice isomorphism.

\bigskip

Let $X$ be a Banach lattice, $\varphi \in \left( X^{\ast }\right) ^{+}\ $and 
$N:=\left\{ x\in X:\varphi \left( \left\vert x\right\vert \right) =0\right\}
.$ Consider the quotient space $X/N.$ Then, the function $\left\Vert \cdot
\right\Vert _{\varphi }:X/N\rightarrow 
\mathbb{R}
$ given by%
\begin{equation*}
\left\Vert \lbrack x]\right\Vert _{\varphi }:=\varphi \left( \left\vert
x\right\vert \right)
\end{equation*}%
is an order-preserving norm. Since the completion of a normed lattice is a
Banach lattice \cite[Theorem 4.2]{Aliprantis}, it follows that the
completion of $X/N$ is a Banach lattice.

\begin{lemma}
Let $X$ be a Banach lattice and $\varphi \in \left( X^{\ast }\right) ^{+}.$
Then the completion of $X/N$ is an abstract $L^{1}$-space.
\end{lemma}

\begin{proof}
Let $\left( X_{1},\left\Vert \cdot \right\Vert _{1}\right) $ be the
completion of $X/N$. Observe that, for every $x,y\in X$ such that $%
[x],[y]\in \left( X/N\right) ^{+}$ we have that \ 
\begin{eqnarray*}
\left\Vert \lbrack x]+[y]\right\Vert _{\varphi } &=&\left\Vert [\left\vert
x\right\vert ]+[\left\vert y\right\vert ]\right\Vert _{\varphi }=\left\Vert
[\left\vert x\right\vert +\left\vert y\right\vert ]\right\Vert _{\varphi
}=\varphi \left( \left\vert x\right\vert +\left\vert y\right\vert \right) \\
&=&\varphi \left( \left\vert x\right\vert \right) +\varphi \left( \left\vert
y\right\vert \right) =\left\Vert [x]\right\Vert _{\varphi }+\left\Vert
[y]\right\Vert _{\varphi }.
\end{eqnarray*}%
\newline
\indent Thus, the norm $\left\Vert \cdot \right\Vert _{\varphi }$ is
additive on $\left( X/N\right) ^{+}.$ Next let $x,y\in X_{1}^{+}.$ Take $%
\left\{ x_{n}\right\} _{n=1}^{\infty },\left\{ y_{n}\right\} _{n=1}^{\infty
}\subset X$ such that $\left\{ [x_{n}]\right\} _{n=1}^{\infty },\left\{
[y_{n}]\right\} _{n=1}^{\infty }\subset \left( X/N\right) ^{+}$ are
convergent in $X^{1}$ to $x$ and $y$, respectively. Then 
\begin{eqnarray*}
\left\Vert x+y\right\Vert _{1} &=&\left\Vert \lim_{n\rightarrow \infty
}[x_{n}]+\lim_{n\rightarrow \infty }[y_{n}]\right\Vert
_{1}=\lim_{n\rightarrow \infty }\left\Vert \left( [x_{n}]+[y_{n}]\right)
\right\Vert _{1} \\
&& \\
&=&\lim_{n\rightarrow \infty }\left\Vert \left( [x_{n}]+[y_{n}]\right)
\right\Vert _{\varphi }=\lim_{n\rightarrow \infty }\left( \left\Vert
[x_{n}]\right\Vert _{\varphi }+\left\Vert [y_{n}]\right\Vert _{\varphi
}\right) \\
&& \\
&=&\lim_{n\rightarrow \infty }\left\Vert [x_{n}]\right\Vert
_{1}+\lim_{n\rightarrow \infty }\left\Vert [y_{n}]\right\Vert
_{1}=\left\Vert x\right\Vert _{1}+\left\Vert y\right\Vert _{1}.
\end{eqnarray*}%
\indent The above proves that $X_{1}$ is an abstract $L_{1}$-space.
\end{proof}

\begin{lemma}
Let $X$ be a Banach lattice and $Y$ a Banach sequence lattice. Then, for
every $n\in 
\mathbb{N}
$, $x_{1},...,x_{n}\in X$ and $\varphi _{1},...,\varphi _{n}\in X^{\ast },%
\label{lemaqcv}$%
\begin{equation}
\sum_{j=1}^{n}\left\vert \varphi _{j}\left( x_{j}\right) \right\vert \leq
\left\langle \left\Vert \left( x_{1},...,x_{n}\right) \right\Vert
_{X^{n},Y,\tau },\left\Vert \left( \varphi _{1},...,\varphi _{n}\right)
\right\Vert _{\left( X^{n}\right) ^{\ast },Y^{\times },\tau }\right\rangle
_{X,X^{\ast }}.  \label{qcv}
\end{equation}
\end{lemma}

\begin{proof}
Let $\varphi _{1},...,\varphi _{n}\in X^{\ast }$. Define 
\begin{equation*}
\varphi :=\left\Vert \left( \varphi _{1},...,\varphi _{n}\right) \right\Vert
_{Y^{\times }}\in \left( X^{\ast }\right) ^{+}\text{ and }N:=\left\{ x\in
X:\varphi \left( \left\vert x\right\vert \right) =0\right\} .
\end{equation*}%
\indent By the above lemma the completion of $\left( X/N,\left\Vert \cdot
\right\Vert _{\varphi }\right) ,$ $X_{1}$ is an abstract $L^{1}$-space. Thus
there exists a localizable measure space $\left( \Omega ,\Sigma ,\mu \right) 
$ and an isometric lattice isomorphism $I:X_{1}\rightarrow L^{1}\left( \mu
\right) .$\newline
\newline
\indent Now for every $j\in \left\{ 1,...,n\right\} \ $let us define the
functional $\widetilde{\varphi _{j}}:X/N\rightarrow 
\mathbb{R}
$ by%
\begin{equation}
\widetilde{\varphi _{j}}\left( [x]\right) :=\varphi _{j}\left( x\right)
,\forall x\in X.  \label{12bn}
\end{equation}%
\indent Observe that%
\begin{equation}
\left\Vert e_{j}\right\Vert _{Y^{\times }}\left\vert \varphi _{j}\right\vert
=\left\Vert \left( 0,...,\varphi _{j},0,...0\right) \right\Vert _{Y^{\times
}}\leq \left\Vert \left( \varphi _{1},...,\varphi _{n}\right) \right\Vert
_{Y^{\times }}=\varphi .  \label{3003}
\end{equation}%
\indent Then, by $\left( \ref{jijiji}\right) ,$%
\begin{equation*}
\left\vert \varphi _{j}\left( z\right) \right\vert \leq \left\vert \varphi
_{j}\right\vert \left( \left\vert z\right\vert \right) \leq \frac{1}{%
\left\Vert e_{j}\right\Vert _{Y^{\times }}}\varphi \left( \left\vert
z\right\vert \right) .
\end{equation*}%
\indent Thus $N$ is contained in the kernel of $\varphi _{j}$ and
consequently $\widetilde{\varphi _{j}}$ is well-defined. Now, since $%
\left\vert \varphi _{j}\right\vert $ is a positive operator, by $\left( \ref%
{3003}\right) $ it follows that 
\begin{equation*}
\left\vert \widetilde{\varphi _{j}}\left( [x]\right) \right\vert =\left\vert
\left\vert \varphi _{j}\right\vert \left( x\right) \right\vert \leq
\left\vert \varphi _{j}\right\vert \left( \left\vert x\right\vert \right)
\leq \frac{1}{\left\Vert e_{j}\right\Vert _{Y^{\times }}}\varphi \left(
\left\vert x\right\vert \right) =\frac{1}{\left\Vert e_{j}\right\Vert
_{Y^{\times }}}\left\Vert [x]\right\Vert _{\varphi },\forall x\in X.
\end{equation*}%
\newline
\indent Therefore $\widetilde{\varphi _{j}}$ is bounded in $X/N$ and in
consequence it is bounded in $X_{1}.$ So, $\widetilde{\varphi _{j}}\in
X_{1}^{\ast }.$\newline
\newline
\indent Now, consider the transpose operator $\left( I^{-1}\right) ^{\ast
}:X_{1}^{\ast }\rightarrow L^{1}\left( \mu \right) ^{\ast }$ and observe
that it is an isometric lattice isomorphism. Fix $x_{1},...,x_{n}\in X$ and
for each $1\leq j\leq n$ let us define%
\begin{equation*}
f_{j}:=I\left( [x_{j}]\right) \in L^{1}\left( \mu \right) ,
\end{equation*}%
\begin{equation*}
\psi _{j}:=\left( I^{-1}\right) ^{\ast }\left( \widetilde{\varphi _{j}}%
\right) \in L^{1}\left( \mu \right) ^{\ast }
\end{equation*}%
and%
\begin{equation*}
g_{j}:=R^{-1}\left( \psi _{j}\right) \in L^{\infty }\left( \mu \right)
\end{equation*}%
where $R:L^{\infty }\left( \mu \right) \rightarrow L^{1}\left( \mu \right)
^{\ast }$ is the canonical map defined in $\left( \ref{juju2}\right) .$
Then, by $\left( \ref{juju2}\right) \ $and $\left( \ref{12bn}\right) $ it
follows that%
\begin{eqnarray*}
\left\vert \varphi _{j}\left( x_{j}\right) \right\vert &=&\left\vert 
\widetilde{\varphi _{j}}\left( \left[ x_{j}\right] \right) \right\vert
=\left\vert \left\langle \left[ x_{j}\right] ,\widetilde{\varphi _{j}}%
\right\rangle \right\vert =\left\vert \left\langle I\left[ x_{j}\right]
,\left( I^{-1}\right) ^{\ast }\widetilde{\varphi _{j}}\right\rangle
\right\vert \\
&& \\
&=&\left\vert \left\langle I\left[ x_{j}\right] ,\left( R\circ R^{-1}\circ
\left( I^{-1}\right) ^{\ast }\right) \widetilde{\varphi _{j}}\right\rangle
\right\vert =\left\vert \left\langle I\left[ x_{j}\right] ,Rg_{j}\right%
\rangle \right\vert \\
&& \\
&=&\left\vert \int_{\Omega }f_{j}\left( w\right) \cdot g_{j}\left( w\right)
dw\right\vert \leq \int_{\Omega }\left\vert f_{j}\left( w\right) \cdot
g_{j}\left( w\right) \right\vert dw.
\end{eqnarray*}%
\newline
\indent Thus, by H\"{o}lder%
\'{}%
s inequality 
\begin{eqnarray}
\sum_{j=1}^{n}\left\vert \varphi _{j}\left( x_{j}\right) \right\vert &\leq &%
\hspace{-0.1cm}\sum_{j=1}^{n}\int_{\Omega }\left\vert f_{j}\left( w\right)
\cdot g_{j}\left( w\right) \right\vert dw=\hspace{-0.1cm}\int_{\Omega
}\sum_{j=1}^{n}\left\vert f_{j}\left( w\right) \cdot g_{j}\left( w\right)
\right\vert dw  \label{fty} \\
&&  \notag \\
&\leq &\int_{\Omega }\left\Vert \left( g_{1}\left( w\right) ,...,g_{n}\left(
w\right) \right) \right\Vert _{Y^{\times }}\cdot \left\Vert \left(
f_{1}\left( w\right) ,...,f_{n}\left( w\right) \right) \right\Vert _{Y}dw 
\notag \\
&&  \notag \\
&=&\hspace{-0.3cm}\int_{\Omega }\left( \left\Vert \left(
g_{1},...,g_{n}\right) \right\Vert _{Y^{\times }}\left( w\right) \right)
\cdot \left( \left\Vert \left( f_{1},...,f_{n}\right) \right\Vert _{Y}\left(
w\right) \right) dw  \notag \\
&&  \notag \\
&=&c,  \notag
\end{eqnarray}%
where $c=\left\langle \left\Vert \left( f_{1},...,f_{n}\right) \right\Vert
_{Y},R\left( \left\Vert \left( g_{1},...,g_{n}\right) \right\Vert
_{Y^{\times }}\right) \right\rangle $. Now, since $I,$ $\left( I^{-1}\right)
^{\ast }\ $and $R$ are order-preserving operators, it follows by Lemma \ref%
{lema1802} and $\left( \ref{12bn}\right) $ that\vspace{-0.1cm}\newline
\begin{eqnarray}
c &=&\left\langle \left\Vert \left( f_{1},...,f_{n}\right) \right\Vert
_{Y},\left\Vert \left( Rg_{1},...,Rg_{n}\right) \right\Vert _{Y^{\times
}}\right\rangle  \label{fty22} \\
&&  \notag \\
&=&\left\langle \left\Vert \left( I\left( [x_{1}]\right) ,...,I\left(
[x_{n}]\right) \right) \right\Vert _{Y},\left\Vert \left( \left(
I^{-1}\right) ^{\ast }\left( \widetilde{\varphi _{1}}\right) ,...,\left(
I^{-1}\right) ^{\ast }\left( \widetilde{\varphi _{n}}\right) \right)
\right\Vert _{Y^{\times }}\right\rangle  \notag \\
&&  \notag \\
&=&\left\langle I\left( \left\Vert \left( [x_{1}],...,[x_{n}]\right)
\right\Vert _{Y}\right) ,\left( I^{-1}\right) ^{\ast }\left( \left\Vert
\left( \widetilde{\varphi _{1}},...,\widetilde{\varphi _{n}}\right)
\right\Vert _{Y^{\times }}\right) \right\rangle  \notag \\
&&  \notag \\
&=&\left\langle \left\Vert \left( [x_{1}],...,[x_{n}]\right) \right\Vert
_{Y},\left\Vert \left( \widetilde{\varphi _{1}},...,\widetilde{\varphi _{n}}%
\right) \right\Vert _{Y^{\times }}\right\rangle  \notag \\
&&  \notag \\
&=&\left\Vert \left( \widetilde{\varphi _{1}}\left( \left\Vert \left(
[x_{1}],...,[x_{n}]\right) \right\Vert _{Y}\right) ,...,\widetilde{\varphi
_{n}}\left( \left\Vert \left( [x_{1}],...,[x_{n}]\right) \right\Vert
_{Y}\right) \right) \right\Vert _{Y^{\times }}  \notag \\
&&  \notag \\
&=&\left\Vert \left( \varphi _{1}\left( \left\Vert \left(
x_{1},...,x_{n}\right) \right\Vert _{Y}\right) ,...,\varphi _{n}\left(
\left\Vert \left( x_{1},...,x_{n}\right) \right\Vert _{Y}\right) \right)
\right\Vert _{Y^{\times }}  \notag \\
&&  \notag \\
&=&\left\langle \left\Vert \left( x_{1},...,x_{n}\right) \right\Vert
_{Y},\left\Vert \left( \varphi _{1},...,\varphi _{n}\right) \right\Vert
_{Y^{\times }}\right\rangle .  \notag
\end{eqnarray}%
\indent From $\left( \ref{fty}\right) $ and $\left( \ref{fty22}\right) $ we
obtain $\left( \ref{qcv}\right) $.
\end{proof}

\bigskip

\bigskip

Next, for each $s=\left\{ \varphi _{n}\right\} _{n=1}^{\infty }\in \ell
^{\tau ,Y^{\times }}\left( X^{\ast }\right) \ $let us define $\varphi
_{s}:\ell _{00}^{\tau ,Y}\left( X\right) \rightarrow 
\mathbb{R}
$ by%
\begin{equation*}
\varphi _{s}\left( \left\{ x_{n}\right\} _{n=1}^{\infty }\right)
=\sum_{n=1}^{\infty }\varphi _{n}\left( x_{n}\right) ,\forall \left\{
x_{n}\right\} _{n=1}^{\infty }\in \ell _{00}^{\tau ,Y}\left( X\right) .
\end{equation*}%
\indent Clearly $\varphi $ is a linear functional. Let $\left\{
x_{n}\right\} _{n=1}^{\infty }\in \ell _{00}^{\tau ,Y}\left( X\right) $ and $%
k\in 
\mathbb{N}
$ be such that $x_{n}=0,n>k.$ Then, by the above lemma%
\begin{eqnarray*}
\left\vert \varphi _{s}\left( \left\{ x_{n}\right\} _{n=1}^{\infty }\right)
\right\vert \hspace{-0.2cm} &=&\hspace{-0.2cm}\left\vert
\sum_{j=1}^{k}\varphi _{j}\left( x_{j}\right) \right\vert \leq
\sum_{j=1}^{k}\left\vert \varphi _{j}\left( x_{j}\right) \right\vert \\
&& \\
&\leq &\left\langle \left\Vert \left( x_{1},...,x_{k}\right) \right\Vert
_{Y},\left\Vert \left( \varphi _{1},...,\varphi _{k}\right) \right\Vert
_{Y^{\times }}\right\rangle _{X,X^{\ast }}. \\
&& \\
&\leq &\left\Vert \left\Vert \left( x_{1},...,x_{k}\right) \right\Vert
_{Y}\right\Vert _{X}\left\Vert \left\Vert \left( \varphi _{1},...,\varphi
_{k}\right) \right\Vert _{Y^{\times }}\right\Vert _{X^{\ast }} \\
&& \\
&=&\left\Vert \left( x_{1},...,x_{k}\right) \right\Vert _{X^{k},\tau
,Y}\left\Vert \left( \varphi _{1},...,\varphi _{k}\right) \right\Vert
_{\left( X^{\ast }\right) ^{k},\tau ,Y^{\times }} \\
&& \\
&\leq &\left\Vert \left\{ x_{n}\right\} _{n=1}^{\infty }\right\Vert _{\ell
^{\tau ,Y}\left( X\right) }\left\Vert s\right\Vert _{\ell ^{\tau ,Y^{\times
}}\left( X^{\ast }\right) }
\end{eqnarray*}%
\indent Thus $\varphi _{s}\in \ell _{00}^{\tau ,Y}\left( X\right) ^{\ast }$
and $\left\Vert \varphi _{s}\right\Vert \leq \left\Vert s\right\Vert _{\ell
^{\tau ,Y^{\times }}\left( X^{\ast }\right) }.$ Therefore $\varphi _{s}$ can
be uniquely extended to $\ell _{0}^{\tau ,Y}\left( X\right) ^{\ast }.$%
\newline
\newline
\indent Let us define $J:\ell ^{\tau ,Y^{\times }}\left( X^{\ast }\right)
\rightarrow \ell _{0}^{\tau ,Y}\left( X\right) ^{\ast }$ by%
\begin{equation}
Js=\varphi _{s},\forall s\in \ell ^{\tau ,Y^{\times }}\left( X^{\ast
}\right) .  \label{dedede2}
\end{equation}%
\indent Observe that $J$ is a continuous linear operator such that%
\begin{equation}
\left\Vert Js\right\Vert _{\ell _{0}^{\tau ,Y}\left( X\right) ^{\ast }}\leq
\left\Vert s\right\Vert _{\ell ^{\tau ,Y^{\times }}\left( X^{\ast }\right)
},\forall s\in \ell ^{\tau ,Y^{\times }}\left( X^{\ast }\right) .
\label{smt}
\end{equation}

\begin{theorem}
Let $X$ be a Banach lattice and $J:\ell ^{\tau ,Y^{\times }}\left( X^{\ast
}\right) \rightarrow \ell _{0}^{\tau ,Y}\left( X\right) ^{\ast }$ as in $%
\left( \ref{dedede2}\right) $. Then $J$ is an isometric isomorphism. We will
denote this by\label{tmapen2}%
\begin{equation}
\ell _{0}^{\tau ,Y}\left( X\right) ^{\ast }=\ell ^{\tau ,Y^{\times }}\left(
X^{\ast }\right) .  \label{30032}
\end{equation}
\end{theorem}

\begin{proof}
We have already proven in $\left( \ref{smt}\right) $ that $J$ is a
continuous linear operator. Now we wil prove that it is surjective and
isometric.\newline
\newline
\indent Given $\varphi \in \ell _{0}^{\tau ,Y}\left( X\right) ^{\ast }$ and $%
n\in 
\mathbb{N}
$ let us define $\varphi _{n}:X\rightarrow 
\mathbb{R}
$ and $\left\vert \varphi \right\vert _{n}:X\rightarrow 
\mathbb{R}
$ by%
\begin{equation*}
\varphi _{n}\left( x\right) :=\varphi \left( e_{n}\cdot x\right) \text{ and }%
\left\vert \varphi \right\vert _{n}\left( x\right) :=\left\vert \varphi
\right\vert \left( e_{n}\cdot x\right) ,\forall x\in X,
\end{equation*}%
where $e_{n}\cdot x=\left( 0,...,x,0,...\right) $. Observe that $\varphi
_{n},\left\vert \varphi \right\vert _{n}\in X^{\ast }$ and \newline
\begin{equation}
\left\vert \varphi _{n}\right\vert \leq \left\vert \varphi \right\vert
_{n},n\in 
\mathbb{N}
.  \label{1303}
\end{equation}%
\newline
\indent Next, in order to establish that $s=\left\{ \varphi _{n}\right\}
_{n=1}^{\infty }\in \ell ^{\tau ,Y^{\times }}\left( X^{\ast }\right) $ we
will prove that%
\begin{equation}
\sup_{n\in 
\mathbb{N}
}\left\{ \left\Vert \left( \varphi _{1},...,\varphi _{n}\right) \right\Vert
_{\left( X^{\ast }\right) ^{n},h_{n,Y^{\times }}}\right\} \leq \left\Vert
\varphi \right\Vert _{\ell _{0}^{\tau ,Y}\left( X\right) ^{\ast }}<\infty .
\label{gfg}
\end{equation}%
\newline
\indent By Corollary \ref{coro1405}, for each $n\in 
\mathbb{N}
$%
\begin{equation}
\left\Vert \left( \varphi _{1},...,\varphi _{n}\right) \right\Vert _{\left(
X^{\ast }\right) ^{n},Y^{\times },\tau }\leq \underset{1\leq j\leq k}{%
\sup_{k\in 
\mathbb{N}
}}\left\{ \left\Vert \bigvee\limits_{j=1}^{k}\sum_{i=1}^{n}a_{i,j}\varphi
_{i}\right\Vert _{X^{\ast }}:\left\Vert \left( a_{1,j},...,a_{n,j}\right)
\right\Vert _{Y}\leq 1.\right\} .  \label{ioio}
\end{equation}%
\indent Let us fix $n,k\in 
\mathbb{N}
$ and $a_{1,j},...,a_{n,j}\in 
\mathbb{R}
$ such that $\left\Vert \left( a_{1,j},...,a_{n,j}\right) \right\Vert
_{Y}\leq 1,1\leq j\leq k.$ Then, by $\left( \ref{1303}\right) $ we have that%
\begin{equation}
\left\vert \bigvee\limits_{j=1}^{k}\sum_{i=1}^{n}a_{i,j}\varphi
_{i}\right\vert \leq \bigvee\limits_{j=1}^{k}\sum_{i=1}^{n}\left\vert
a_{i,j}\right\vert \left\vert \varphi _{i}\right\vert \leq
\bigvee\limits_{j=1}^{k}\sum_{i=1}^{n}\left\vert a_{i,j}\right\vert
\left\vert \varphi \right\vert _{i}.  \label{op2}
\end{equation}%
\indent Since $\bigvee\limits_{j=1}^{k}\sum_{i=1}^{n}\left\vert
a_{i,j}\right\vert \left\vert \varphi \right\vert _{i}$ is a positive
functional it follows from $\left( \ref{positive}\right) $ and $\left( \ref%
{defi1802}\right) $ that, for each $x\in X$%
\begin{eqnarray}
\left\vert \bigvee\limits_{j=1}^{k}\sum_{i=1}^{n}\left\vert
a_{i,j}\right\vert \left\vert \varphi \right\vert _{i}\left( x\right)
\right\vert &\leq &\left( \bigvee\limits_{j=1}^{k}\sum_{i=1}^{n}\left\vert
a_{i,j}\right\vert \left\vert \varphi \right\vert _{i}\right) \left(
\left\vert x\right\vert \right)  \label{koko} \\
&&  \notag \\
&=&\sup \left\{ \sum_{j=1}^{k}\left( \sum_{i=1}^{n}\left\vert
a_{i,j}\right\vert \left\vert \varphi \right\vert _{i}\right) \left(
x_{j}\right) \right\}  \notag \\
&&  \notag \\
&=&\sup \left\{ \left\vert \varphi \right\vert \left( \left\{
\sum_{j=1}^{k}\left\vert a_{m,j}\right\vert x_{j}\right\} _{m=1}^{n}\right)
\right\} ,  \notag
\end{eqnarray}%
where the supremum is taken over all $x_{j}\geq 0$ such that $%
\sum_{j=1}^{k}x_{j}=\left\vert x\right\vert .$\newline
\newline
\indent Take $x_{1},...,x_{k}\geq 0$ such that $\sum_{j=1}^{k}x_{j}=\left%
\vert x\right\vert $ and let us denote $\left\Vert \left\vert \varphi
\right\vert \right\Vert _{\ell _{0}^{\tau ,Y}\left( X\right) ^{\ast }}$ by $%
c.$ Then, by Lemma \ref{lema23052} 
\begin{eqnarray}
\left\vert \varphi \right\vert \left( \left\{ \sum_{j=1}^{k}\left\vert
a_{m,j}\right\vert x_{j}\right\} _{m=1}^{n}\right) \hspace{-0.2 cm} &\leq
&c\left\Vert \left( \sum_{j=1}^{k}\left\vert a_{1,j}\right\vert
x_{j},...,\sum_{j=1}^{k}\left\vert a_{n,j}\right\vert x_{j}\right)
\right\Vert _{\ell _{0}^{\tau ,Y}\left( X\right) }  \label{fin1} \\
&&  \notag \\
&=&c\left\Vert \left( \sum_{j=1}^{k}\left\vert a_{1,j}\right\vert
x_{j},...,\sum_{j=1}^{k}\left\vert a_{n,j}\right\vert x_{j}\right)
\right\Vert _{X^{n},Y,\tau }  \notag \\
&&  \notag \\
&=&c\left\Vert \left\Vert \left( \sum_{j=1}^{k}\left\vert a_{1,j}\right\vert
x_{j},...,\sum_{j=1}^{k}\left\vert a_{n,j}\right\vert x_{j}\right)
\right\Vert _{Y}\right\Vert _{X}  \notag \\
&&  \notag \\
&\leq &c\left\Vert \sum_{j=1}^{k}\left\Vert \left( \left\vert
a_{1,j}\right\vert x_{j},...,\left\vert a_{n,j}\right\vert x_{j}\right)
\right\Vert _{Y}\right\Vert _{X}.  \notag
\end{eqnarray}%
\newline
\indent Now let us define $G:%
\mathbb{R}
^{n}\rightarrow 
\mathbb{R}
^{n}$ by $G\left( t_{1},...,t_{n}\right) =\left( \left\vert
a_{1,j}\right\vert t_{1},...,\left\vert a_{n,j}\right\vert t_{n}\right) .$
Observe that $G\in \mathcal{H}_{n}^{n}$ and consequently, by $\left( \ref%
{smt2}\right) $, $\left\Vert \cdot \right\Vert _{%
\mathbb{R}
^{n},Y}\circ G\in \mathcal{H}_{n}.$ Then it follows from Lemma \ref{lematmm}
and Proposition \ref{propkrivi} that%
\begin{eqnarray}
\left\Vert \left( \left\vert a_{1,j}\right\vert x_{j},...,\left\vert
a_{n,j}\right\vert x_{j}\right) \right\Vert _{Y} &=&\tau _{\left( \left(
\left\vert a_{1,j}\right\vert x_{j},...,\left\vert a_{n,j}\right\vert
x_{j}\right) \right) }\left( \left\Vert \cdot \right\Vert _{%
\mathbb{R}
^{n},Y}\right)  \label{fin2} \\
&&  \notag \\
&=&\tau _{\left( x_{j},...,x_{j}\right) }\left( \left\Vert \cdot \right\Vert
_{%
\mathbb{R}
^{n},Y}\circ G\right)  \notag \\
&&  \notag \\
&\leq &\left( \left\vert x_{j}\right\vert \vee ...\vee \left\vert
x_{j}\right\vert \right) \left\Vert \left( \left\Vert \cdot \right\Vert _{%
\mathbb{R}
^{n},Y}\circ G\right) \right\Vert _{\mathcal{H}_{n}}  \notag \\
&&  \notag \\
&=&\left\vert x_{j}\right\vert \underset{\left( t_{1},...,t_{n}\right) \in
S^{n-1}}{\sup }\left\{ \left\Vert \left( \left\vert a_{1,j}\right\vert
t_{1},...,\left\vert a_{n,j}\right\vert t_{n}\right) \right\Vert _{Y}\right\}
\notag \\
&&  \notag \\
&\leq &\left\vert x_{j}\right\vert =x_{j}.  \notag
\end{eqnarray}%
\indent Then, due to $\left( \ref{fin1}\right) $ and $\left( \ref{fin2}%
\right) $ 
\begin{equation}
\left\vert \varphi \right\vert \left( \sum_{j=1}^{k}\left\vert
a_{1,j}\right\vert x_{j},...,\sum_{j=1}^{k}\left\vert a_{n,j}\right\vert
x_{j}\right) \leq \left\Vert \varphi \right\Vert _{\ell _{0}^{\tau ,Y}\left(
X\right) ^{\ast }}\left\Vert \sum_{j=1}^{k}x_{j}\right\Vert _{X}=\left\Vert
\varphi \right\Vert _{\ell _{0}^{\tau ,Y}\left( X\right) ^{\ast }}\left\Vert
x\right\Vert _{X}  \label{koko1}
\end{equation}%
\indent Thus, from $\left( \ref{koko}\right) $ and $\left( \ref{koko1}%
\right) $%
\begin{equation*}
\left\Vert \bigvee\limits_{j=1}^{k}\sum_{i=1}^{n}\left\vert
a_{i,j}\right\vert \left\vert \varphi \right\vert _{i}\right\Vert _{X^{\ast
}}=\sup \left\{ \bigvee\limits_{j=1}^{k}\sum_{i=1}^{n}\left\vert
a_{i,j}\right\vert \left\vert \varphi \right\vert _{i}\left( x\right) :x\in
B_{X}\right\} \leq \left\Vert \varphi \right\Vert _{\ell _{0}^{\tau
,Y}\left( X\right) ^{\ast }}.
\end{equation*}%
\indent Therefore $\left( \ref{gfg}\right) $ is satisfied and consequently $%
s\in \ell ^{\tau ,Y^{\times }}\left( X^{\ast }\right) .$ \newline
\newline
\indent On the other hand for each $x\in \left\{ x_{n}\right\}
_{n=1}^{\infty }\in \ell _{00}^{\tau ,Y}\left( X\right) ,$%
\begin{equation*}
Js\left( x\right) =\varphi _{s}\left( x\right) =\sum_{n=1}^{\infty }\varphi
_{n}\left( x_{n}\right) =\sum_{n=1}^{\infty }\varphi \left( e_{n}\cdot
x_{n}\right) =\varphi \left( \sum_{n=1}^{\infty }e_{n}\cdot x_{n}\right)
=\varphi \left( x\right) .
\end{equation*}%
\indent Then $Js=\varphi $ in $\ell _{00}^{\tau ,Y}\left( X\right) $ and
consequently in $\ell _{0}^{\tau ,Y}\left( X\right) .$ So $J$ is surjective.
Finally it follows from $\left( \ref{gfg}\right) $ that $\left\Vert
Js\right\Vert _{\ell _{0}^{\tau ,Y}\left( X\right) ^{\ast }}\geq \left\Vert
s\right\Vert _{\ell ^{\tau ,Y^{\times }}\left( X^{\ast }\right) }$ proving
this way that $J$ is an isometry and $\left( \ref{30032}\right) $ is
satisfied$.$
\end{proof}

\section{Duality between $Y$-convexity and $Y^{\times }$-concavity}

\begin{theorem}
Let $X$ be a Banach lattice and $E$ a Banach space.\label{thmdualidad}%
.\smallskip \newline
\indent$i)$ Let $T\in \mathcal{L}\left( E,X\right) $. Then $T$ is $Y$-convex
if and only if $T^{\ast }:X^{\ast }\rightarrow E^{\ast }$ is $Y^{\times }$%
-concave. In this case, $\left\Vert T\right\Vert _{K^{Y}}=\left\Vert T^{\ast
}\right\Vert _{K_{Y^{\times }}}.$\smallskip \newline
\indent$ii)$ Let $S\in \mathcal{L}\left( X,E\right) .$ Then $S\ $is $Y$%
-concave if and only if $S^{\ast }:E^{\ast }\rightarrow X^{\ast }$ is $%
Y^{\times }$-convex. In this case, $\left\Vert S^{\ast }\right\Vert
_{K^{Y}}=\left\Vert S\right\Vert _{K_{Y^{\times }}}.$\smallskip \newline
\end{theorem}

\begin{proof}
Since $T\in \mathcal{L}\left( E,X\right) $ we have that $T^{\ast }\in 
\mathcal{L}\left( X^{\ast },E^{\ast }\right) .$ For each $n\in 
\mathbb{N}
$ consider the operators%
\begin{equation*}
\left( T^{\ast }\right) _{n}:\left( \left( X^{\ast }\right) ^{n},\left\Vert
\cdot \right\Vert _{\left( X^{\ast }\right) ^{n},Y^{\times },\tau }\right)
\rightarrow \left( \left( E^{\ast }\right) ^{n},\left\Vert \cdot \right\Vert
_{\left( E^{\ast }\right) ^{n},Y^{\times }}\right)
\end{equation*}%
and%
\begin{equation*}
\left( T_{n}\right) ^{\ast }:\left( X^{n},\left\Vert \cdot \right\Vert
_{X^{n},Y,\tau }\right) ^{\ast }\rightarrow \left( E^{n},\left\Vert \cdot
\right\Vert _{E^{n},Y}\right) ^{\ast }.
\end{equation*}%
\indent By Lemma \ref{nnlema} and Proposition \ref{propoca} these operators
are well defined and continuous. \newline
\newline
\indent Now consider the operators $\rho _{n}:\left( \left( E^{\ast }\right)
^{n},\left\Vert \cdot \right\Vert _{\left( E^{\ast }\right) ^{n},Y^{\times
}}\right) \rightarrow \left( E^{n},\left\Vert \cdot \right\Vert
_{E^{n},Y}\right) ^{\ast }$ and $J_{n}:\left( \left( X^{\ast }\right)
^{n},\left\Vert \cdot \right\Vert _{\left( X^{\ast }\right) ^{n},Y^{\times
},\tau }\right) \rightarrow \left( X^{n},\left\Vert \cdot \right\Vert
_{X^{n},Y,\tau }\right) ^{\ast }$ given by the restriction of operators $%
\rho $ and $J$ defined in $\left( \ref{line}\right) $ and $\left( \ref%
{dedede2}\right) $ respectively. Next we will prove that 
\begin{equation}
\rho _{n}\circ \left( T^{\ast }\right) _{n}=\left( T_{n}\right) ^{\ast
}\circ J_{n}.  \label{wnb}
\end{equation}%
\indent Let $s=\left\{ \varphi _{j}\right\} _{j=1}^{n}\in \left( X^{\ast
}\right) ^{n}.$ Then the functional $\left( T_{n}\right) ^{\ast }\circ
J_{n}\left( s\right) \in \left( E^{n},\left\Vert \cdot \right\Vert
_{Y}\right) ^{\ast }$ satisfies that, for each $\left\{ w_{j}\right\}
_{j=1}^{n}\in E^{n}$%
\begin{eqnarray*}
\left\langle \left\{ w_{j}\right\} _{j=1}^{n},\left( T_{n}\right) ^{\ast
}\circ J_{n}\left( s\right) \right\rangle _{E^{n}\times \left( E^{n}\right)
^{\ast }} &=&\left\langle T_{n}\left( \left\{ w_{j}\right\}
_{j=1}^{n}\right) ,J_{n}\left( s\right) \right\rangle _{X^{n}\times \left(
X^{n}\right) ^{\ast }} \\
&& \\
&=&\left\langle \left( \left\{ Tw_{j}\right\} _{j=1}^{n}\right) ,J_{n}\left(
s\right) \right\rangle _{X^{n}\times \left( X^{n}\right) ^{\ast }} \\
&& \\
&=&\sum_{j=1}^{n}\varphi _{j}\left( Tw_{j}\right)
=\sum_{j=1}^{n}\left\langle w_{j},T^{\ast }\varphi _{j}\right\rangle
_{E\times E^{\ast }} \\
&& \\
&=&\left\langle \left\{ w_{j}\right\} _{j=1}^{n},\rho _{n}\left( \left\{
T^{\ast }\varphi _{j}\right\} _{j=1}^{n}\right) \right\rangle _{E^{n}\times
\left( E^{n}\right) ^{\ast }} \\
&& \\
&=&\left\langle \left\{ w_{j}\right\} _{j=1}^{n},\rho _{n}\circ \left(
T^{\ast }\right) _{n}\left( s\right) \right\rangle _{E^{n}\times \left(
E^{n}\right) ^{\ast }}.
\end{eqnarray*}%
\indent Therefore $\left( \ref{wnb}\right) $ is satisfied. Since $\rho $ and 
$J$ are isometric linear operators, it follows that $\rho _{n}$ and $J_{n}$
are also isometries and consequently%
\begin{equation*}
\left\Vert \left( T^{\ast }\right) _{n}\right\Vert =\left\Vert \left(
T_{n}\right) ^{\ast }\right\Vert .
\end{equation*}%
\newline
\indent$i)$ Now, assume that $T$ is $Y$-convex. Then, for each $n\in 
\mathbb{N}
$ and $w_{1},...,w_{n}\in E$ 
\begin{equation*}
\left\Vert \left( T_{n}\left( \left\{ w_{j}\right\} _{j=1}^{n}\right)
\right) \right\Vert _{X^{n},Y,\tau }=\left\Vert \left(
Tw_{1},...,Tw_{n}\right) \right\Vert _{X^{n},Y,\tau }\leq \left\Vert
T\right\Vert _{K^{Y}}\left\Vert \left\{ w_{j}\right\} _{j=1}^{n}\right\Vert
_{E^{n},Y}.
\end{equation*}%
\indent Thus 
\begin{equation*}
\left\Vert \left( T^{\ast }\right) _{n}\right\Vert =\left\Vert \left(
T_{n}\right) ^{\ast }\right\Vert =\left\Vert T_{n}\right\Vert \leq
\left\Vert T\right\Vert _{K^{Y}}
\end{equation*}%
and consequently, for every $\varphi _{1},...,\varphi _{n}\in X^{\ast }$%
\begin{eqnarray*}
\left\Vert \left( T^{\ast }\varphi _{1},...,T^{\ast }\varphi _{n}\right)
\right\Vert _{\left( E^{\ast }\right) ^{n},Y} &=&\left\Vert \left( \left(
T^{\ast }\right) _{n}\left( \left\{ \varphi _{j}\right\} _{j=1}^{n}\right)
\right) \right\Vert _{\left( E^{\ast }\right) ^{n},Y} \\
&\leq &\left\Vert T\right\Vert _{\mathcal{K}^{Y}}\left\Vert \left( \left\{
\varphi _{j}\right\} _{j=1}^{n}\right) \right\Vert _{\left( X^{\ast }\right)
^{n},Y^{\times },\tau }.
\end{eqnarray*}%
\indent We conclude that $T^{\ast }$ is $Y^{\times }$-concave and 
\begin{equation}
\left\Vert T^{\ast }\right\Vert _{\mathcal{K}_{Y^{\times }}}\leq \left\Vert
T\right\Vert _{\mathcal{K}^{Y}}.\newline
\label{109}
\end{equation}%
\newline
\indent On the other hand suppose that $T^{\ast }$ is $Y^{\times }$-concave.
Then, for every $n\in 
\mathbb{N}
$ and $\varphi _{1},...,\varphi _{n}\in X^{\ast },$ 
\begin{eqnarray*}
\left\Vert \left( \left( T^{\ast }\right) _{n}\left( \left\{ \varphi
_{j}\right\} _{j=1}^{n}\right) \right) \right\Vert _{\left( E^{\ast }\right)
^{n},Y} &=&\left\Vert \left( T^{\ast }\varphi _{1},...,T^{\ast }\varphi
_{n}\right) \right\Vert _{\left( E^{\ast }\right) ^{n},Y} \\
&\leq &\left\Vert T^{\ast }\right\Vert _{K_{Y^{\times }}}\left\Vert \left(
\left\{ \varphi _{j}\right\} _{j=1}^{n}\right) \right\Vert _{\left( X^{\ast
}\right) ^{n},Y^{\times },\tau }.
\end{eqnarray*}%
\indent Thus 
\begin{equation*}
\left\Vert T_{n}\right\Vert =\left\Vert \left( T_{n}\right) ^{\ast
}\right\Vert =\left\Vert \left( T^{\ast }\right) _{n}\right\Vert \leq
\left\Vert T^{\ast }\right\Vert _{K_{Y^{\times }}}
\end{equation*}%
and consequently, for every $w_{1},...,w_{n}\in E,$%
\begin{equation*}
\left\Vert \left( Tw_{1},...,Tw_{n}\right) \right\Vert _{X^{n},Y,\tau
}=\left\Vert \left( T_{n}\left( \left\{ w_{j}\right\} _{j=1}^{n}\right)
\right) \right\Vert _{X^{n},Y,\tau }\leq \left\Vert T^{\ast }\right\Vert
_{K_{Y^{\times }}}\left\Vert \left\{ w_{j}\right\} _{j=1}^{n}\right\Vert
_{E^{n},Y}.
\end{equation*}%
\indent We conclude that $T$ is $Y$-convex and 
\begin{equation}
\left\Vert T\right\Vert _{\mathcal{K}^{Y}}\leq \left\Vert T^{\ast
}\right\Vert _{K_{Y^{\times }}}.  \label{111}
\end{equation}%
\indent Finally, it follows from $\left( \ref{109}\right) $ and $\left( \ref%
{111}\right) $ that 
\begin{equation*}
\left\Vert T\right\Vert _{\mathcal{K}^{Y}}=\left\Vert T^{\ast }\right\Vert
_{K_{Y^{\times }}}.
\end{equation*}%
\newline
\indent$ii)$ This proof is analogous to the previous one.
\end{proof}

\bigskip

By taking $E=X$ and $T,S$ as the identity maps in the above theorem we
obtain the following corollary.

\begin{corollary}
Let $X$ be a Banach lattice and $Y$ a Banach sequence lattice.
Then\smallskip \newline
\indent$i)$ $X$ is $Y$-convex if and only if $X^{\ast }$ is $Y^{\times }$%
-concave . In this case, $M_{Y^{\times }}\left( X^{\ast }\right)
=M^{Y}\left( X\right) .$\smallskip \newline
\indent$ii)$ $X$ is $Y$-concave if and only if $X^{\ast }$ is $Y^{\times }$%
-convex. In this case, $M^{Y^{\times }}\left( X^{\ast }\right) =M_{Y}\left(
X\right) .$
\end{corollary}

\end{document}